\documentclass{article}
\usepackage{fullpage}
\usepackage{amsmath}
\usepackage{amsthm}
\usepackage{amssymb}
\usepackage{mathtools}
\usepackage{caption}
\usepackage[labelformat=simple]{subcaption}
\usepackage[tmargin=1in, bmargin=.75in, lmargin=.75in, rmargin = .75in]{geometry}

\usepackage{amsthm}
\usepackage{amsmath}

\usepackage{letltxmacro}                         
\LetLtxMacro\amsproof\proof                      
\LetLtxMacro\amsendproof\endproof                
\usepackage{thmtools}                            %
\AtBeginDocument{
  \LetLtxMacro\proof\amsproof                    %
  \LetLtxMacro\endproof\amsendproof              %
}                                                %

\usepackage{thm-restate}
\usepackage[colorlinks=true, allcolors=NavyBlue]{hyperref}
\usepackage{
  hyperref,
  cleveref,
}
\usepackage[dvipsnames]{xcolor} 

\counterwithin*{COUNTERHACKSUB}{COUNTERHACK}

\newcommand{\MyThmtoolsConstructor}[3]{
  \declaretheorem[
  numberlike=COUNTERHACK,
  numbered=yes,
  name=#2,
  Refname={#2}, 
  refname={\MakeLowercase{#2}},  
  shaded={bgcolor=#3, margin=0pt, padding=5pt},
  ]{#1}
  \declaretheorem[
  numbered=no,
  name=Informal #2,
  Refname={#2}, 
  refname={\MakeLowercase{#2}},  
  ]{#1-informal}
  \declaretheorem[
  numbered=no,
  name=#2,
  Refname={#2}, 
  refname={\MakeLowercase{#2}},  
  ]{#1*}
  \declaretheorem[
  numbered=yes,
  numberlike=COUNTERHACK,
  name=#2,
  Refname={#2}, 
  refname={\MakeLowercase{#2}},  
  style=plain,
  shaded={
    rulecolor=Black,  
    rulewidth=0.5pt, 
    margin=5pt, 
    bgcolor=white
    }
  ]{#1-known}
}

\newcommand{\MySubThmtoolsConstructor}[3]{
  \declaretheorem[
  numberlike=COUNTERHACKSUB,
  numbered=yes,
  name=#2,
  Refname={#2}, 
  refname={\MakeLowercase{#2}},  
  ]{#1}
  \declaretheorem[
  numbered=no,
  name=Informal #2,
  Refname={#2}, 
  refname={\MakeLowercase{#2}},  
  ]{#1-informal}
  \declaretheorem[
  numbered=no,
  name=#2,
  Refname={#2}, 
  refname={\MakeLowercase{#2}},  
  ]{#1*}
}

\newcommand{\MyThmtoolsConstructorRemarkNoBackground}[3]{
  \declaretheorem[
  style=remark,
  numberlike=COUNTERHACK,
  numbered=yes,
  name=#2,
  Refname={#2}, 
  refname={\MakeLowercase{#2}},  
  ]{#1}
  \declaretheorem[
  style=remark,
  numbered=no,
  name=Informal #2,
  Refname={#2}, 
  refname={\MakeLowercase{#2}},  
       shaded={bgcolor=Gray!10, margin=0pt, padding=5pt},
  ]{#1-informal}
  \declaretheorem[
  style=remark,
  numbered=no,
  name=#2,
  Refname={#2}, 
  refname={\MakeLowercase{#2}},  
       shaded={bgcolor=Gray!35, margin=0pt, padding=5pt},
  ]{#1*}
  \declaretheorem[
  style=remark,
  numberlike=COUNTERHACK,
  numbered=yes,
  name=#2,
  Refname={#2}, 
  refname={\MakeLowercase{#2}},  
  ]{#1-breakable}
  \declaretheorem[
  style=remark,
  numbered=no,
  name=Informal #2,
  Refname={#2}, 
  refname={\MakeLowercase{#2}},  
  ]{#1-informal-breakable}
  \declaretheorem[
  style=remark,
  numbered=no,
  name=#2,
  Refname={#2}, 
  refname={\MakeLowercase{#2}},  
  ]{#1-breakable*}
}

\MyThmtoolsConstructor{definition}   {Definition}     {Gray!20}
\MyThmtoolsConstructor{theorem}      {Theorem}        {Gray!20}
\MyThmtoolsConstructor{temporary}    {Example}      {Gray!20}
\MyThmtoolsConstructor{corollary}    {Corollary}      {Gray!20}
\MyThmtoolsConstructor{porism}       {Porism}         {Gray!20}
\MyThmtoolsConstructor{lemma}        {Lemma}          {Gray!20}
\MyThmtoolsConstructor{proposition}  {Proposition}    {Gray!20}
\MyThmtoolsConstructor{comment}      {Comment}        {Gray!20}
\MyThmtoolsConstructor{exercise}     {Exercise}       {Gray!20}
\MyThmtoolsConstructor{conjecture}   {Conjecture}     {Gray!20}
\MyThmtoolsConstructor{reflection}   {Reflection}     {Gray!20}
\MyThmtoolsConstructor{question}     {Question}       {Gray!20}
\MyThmtoolsConstructor{observation} {Observation}     {Gray!20}
\MyThmtoolsConstructor{claim}        {Claim}          {Gray!20}
\MyThmtoolsConstructor{fact}         {Fact}           {Gray!20}
\MyThmtoolsConstructor{note}         {Note}           {Gray!20}
\MyThmtoolsConstructor{notation}     {Notation}       {Gray!20}
\MyThmtoolsConstructor{summary}      {Summary Result} {Gray!20}
\MyThmtoolsConstructor{problem}      {Problem}        {Gray!20}
\MyThmtoolsConstructor{statement}      {Statement}        {Gray!20}

\MyThmtoolsConstructorRemarkNoBackground{remark}       {Remark}         {Gray!20}
\AtBeginEnvironment{remark}{%
\pushQED{\qed}%
}
\AtEndEnvironment{remark}{\popQED\endremark}

\MySubThmtoolsConstructor{subclaim}     {Claim}

\declaretheorem[
    name=Note,
    numbered=no,
    style=remark,
    shaded={
        rulecolor=NavyBlue,  
        rulewidth=1pt, 
        margin=5pt, 
        bgcolor=white
    }
]{readernote}

\newcommand*{\subclaimproofname}{Proof of Claim}
\newenvironment{subclaimproof}[1][\subclaimproofname]{\begin{proof}[#1]}{\end{proof}}
\usepackage{todonotes}
\usepackage{tikz}
\usepackage{tikz-3dplot}

\usepackage{graphicx}
\usepackage{animate}
\usepackage{caption}
\usepackage{complexity}
\usepackage{changepage} 
\usepackage{lipsum}
\usepackage{stackengine} 
\usepackage{graphicx} 
\usepackage{thmbox} 
\usepackage{enumitem}

\newcommand{\clos}{\overline}

\newcommand{\N}{\mathbb{N}}     
\renewcommand{\R}{\mathbb{R}}     

\renewcommand{\P}{\mathcal{P}}  
\renewcommand{\C}{\mathcal{C}}  

\renewcommand{\A}{\mathcal{A}}

\newcommand{\ballo}[2]{B^{\circ}(#1,#2)}
\newcommand{\ballc}[2]{\clos{B}(#1,#2)}
\newcommand{\ballon}[3]{B^{\circ}_{#3}(#1,#2)} 
\newcommand{\ballcn}[3]{\clos{B}_{#3}(#1,#2)} 

\newcommand{\balloinf}[2]{\ballon{#1}{#2}{\infty}} 
\newcommand{\ballcinf}[2]{\ballcn{#1}{#2}{\infty}} 
\newcommand{\ballostd}[2]{\ballon{#1}{#2}{{\scriptscriptstyle\norm{\cdot}}}} 
\newcommand{\ballcstd}[2]{\ballcn{#1}{#2}{{\scriptscriptstyle\norm{\cdot}}}} 

\newcommand{\unitballmeasurestd}{v_{{\scriptscriptstyle \norm{\cdot}},d}}
\newcommand{\unitballmeasureinf}{v_{{\scriptscriptstyle \norm{\cdot}_\infty},d}}

\DeclareMathOperator{\range}{range}
\DeclareMathOperator{\defeq}{\overset{def}{=}}

\DeclarePairedDelimiter{\ceil}{\lceil}{\rceil}
\DeclarePairedDelimiter\floor{\lfloor}{\rfloor}
\DeclarePairedDelimiter\brackets{[}{]}
\DeclarePairedDelimiter\set{\lbrace}{\rbrace}
\DeclarePairedDelimiter\abs{\lvert}{\rvert}
\DeclarePairedDelimiter\norm{\lVert}{\rVert}

\makeatletter
\let\oldabs\abs
\def\abs{\@ifstar{\oldabs}{\oldabs*}}

\let\oldnorm\norm
\def\norm{\@ifstar{\oldnorm}{\oldnorm*}}

\let\oldceil\ceil
\def\ceil{\@ifstar{\oldceil}{\oldceil*}}

\let\oldfloor\floor
\def\floor{\@ifstar{\oldfloor}{\oldfloor*}}

\let\oldbrackets\brackets
\def\brackets{\@ifstar{\oldbrackets}{\oldbrackets*}}

\makeatother

\newcommand*{\numberfullref}[1]{\hyperref[{#1}]{\Autoref*{#1} (\nameref*{#1})}}
\newcommand*{\namefullref}[1]{\hyperref[{#1}]{\nameref*{#1} (\Autoref*{#1})}}

\renewcommand{\epsilon}{\varepsilon}
\input{_prelude/_authors}

\emergencystretch=\maxdimen
\hypersetup{pdfauthor={Jason Vander Woude, Peter Dixon, A. Pavan, Jamie Radcliffe, and N. V. Vinodchandran}}

\hypersetup{pdftitle={Neighborhood Variants of the KKM Lemma, Lebesgue Covering Theorem, and Sperner's Lemma on the Cube}}

\hypersetup{pdfsubject={Discrete Mathematics}}

\hypersetup{pdfkeywords={Sperner's lemma, KKM lemma, Lebesgue covering theorem, Lebesgue covering dimension, cube, hypercube}}

\title{Neighborhood Variants of the KKM Lemma,\\Lebesgue Covering Theorem, and Sperner's Lemma on the Cube}

\begin{document}
\maketitle

\begin{abstract}
We establish a ``neighborhood'' variant of the cubical KKM lemma and the Lebesgue covering theorem and deduce a discretized version which is a ``neighborhood'' variant of Sperner's lemma on the cube. The main result is the following: for any coloring of the unit $d$-cube $[0,1]^d$ in which points on opposite faces must be given different colors, and for any $\epsilon>0$, there is an $\ell_\infty$ $\epsilon$-ball which contains points of at least $(1+\frac{\epsilon}{1+\epsilon})^d$ different colors, (so in particular, at least $(1+\frac23\epsilon)^d$ different colors for all sensible $\epsilon\in(0,\frac12]$).
\end{abstract}

\vfill
\tableofcontents
\vfill

\newpage

\begin{readernote}
The main result of this paper (\numberfullref{kkm-lebesgue-variant}) along with most contents of the appendices originally appeared in a theoretical computer science paper of ours currently in the computer science category of ArXiv \cite{FOCS23_submission}. The purpose of the present article is to highlight this result on its own and make it more visible to the mathematical community. It is reproduced here along with a detailed motivation, additional discussion, and a discretized version of the result that does not appear in \cite{FOCS23_submission}. We note, however, that \cite{FOCS23_submission} contains additional results  including an analogous result for $\R^d$ (as compared to $[0,1]^d$ here) which will not be discussed in this paper; in that context a slightly better bound of $(1+2\epsilon)^d$ (as compared to $(1+\frac23\epsilon)^d$ here) can be achieved. Furthermore, \cite{FOCS23_submission} includes results for {\em every norm} on $\R^d$ (not just $\ell_\infty$), but this is less natural on $[0,1]^d$, so in this paper we focus only on the $\ell_\infty$ norm.
\end{readernote}
\vspace{0.25in}

\section{Introduction}

The Lebesgue covering theorem (see \cite{lebesgue_sur_1911} or \cite[Theorem~IV~2]{dimension_theory_lebesgue_covering}),  Sperner's lemma on the cube (see \cite{de_loera_polytopal_2002}), and the KKM lemma on the cube (see \cite{Kuhn66, wolsey_cubical_1977, komiya_simple_1994, van_der_laan_intersection_1999}) are all known to be naturally equivalent in that any of the three results can be used to fairly directly prove any of the others. Informally, they guarantee that in any well-behaved coloring/covering of the $d$-dimensional cube, there exists a point at the closure of at least $d+1$ colors/sets. We prove the following ``neighborhood'' variant(s) of these results by considering an open ball instead of a point: {\em for any well-behaved coloring/covering of the $d$-dimensional cube and any $\epsilon>0$, there exists a placement of the $\ell_\infty$ $\epsilon$-ball (i.e. a cube of side length $2\epsilon$) which intersects at least $(1+\frac{\epsilon}{1+\epsilon})^d$ colors/sets.} Thus, while the traditional results give a linear bound ($d+1$) on the number of colors/sets, for any fixed $\epsilon$, our neighborhood variant gives an exponential (in $d$) bound on the number of colors/sets.

We first state a theorem which is naturally equivalent\footnote{
    The equivalence is discussed in \Autoref{subsec:kkm-lebesgue} and proved in \Autoref{subsec:kkm-lebesgue-coloring-equivalence}.
} to both the 
cubical KKM lemma 
and the 
Lebesgue covering theorem 
but stated in terms of colorings to be more convenient to work with (\Autoref{kkm-lebesgue}). Then we formally state our main result (\Autoref{kkm-lebesgue-variant}) which is a neighborhood variant of \Autoref{kkm-lebesgue}. As direct corollaries of \Autoref{kkm-lebesgue-variant}, we obtain neighborhood variants of both the 
cubical KKM lemma 
and the 
Lebesgue covering theorem.
We also demonstrate that under an appropriate perspective, our main result implies a neighborhood variant of Sperner's lemma on the cube.


Below, we consider $[0,1]^d$ to be the standard unit $d$-cube and $V=\set{0,1}^{d}$ to be its set of vertices. A face $F$ of the cube $[0,1]^d$ is a product set $F=\prod_{i=1}^d F_i$ where each $F_i$ is one of three sets: $\set{0}$, $\set{1}$, or $[0,1]$. 
For a set $X\subseteq\R^d$ and coordinate $i\in[d]$, we will use the projection notation $\pi_i(X)=\set{x_i:\vec{x}\in X}$.
Two faces $F,F'$ are said to be opposite each other if there is some coordinate $i_0\in[d]$ such that $F_{i_0}=\set{0}$ and $F_{i_0}'=\set{1}$ (or vice versa). We will frequently say that a pair of points $\vec{x},\vec{x}\,'\in[0,1]^d$ belong to opposite faces if there is an opposite pair of faces $F,F'$ with $\vec{x}\in F$ and $\vec{x}\,'\in F'$. Equivalently, a pair of points $\vec{x},\vec{x}\,'\in[0,1]^d$ belong to opposite faces if there is some coordinate $i_0\in[d]$ such that $x_{i_0}=0$ and $x'_{i_0}=1$ (or vice versa). A final equivalent characterization is that a pair of points $\vec{x},\vec{x}\,'\in[0,1]^d$ belong to opposite faces if $\norm{\vec{x}-\vec{x}\,'}_\infty=1$.
We will say that a set $S$ does not contain points of opposite faces if for every pair of opposite faces $F,F'$ at least one of $S\cap F$ and $S\cap F'$ is empty. Equivalently, $S$ does not contain points of opposite faces if no pair of points from $S$ are $\ell_\infty$ distance $1$ apart.


\begin{definition}[Sperner-Lebesgue-Knaster-Kuratowski-Mazurkiewicz Coloring]
For a set $\Lambda\subseteq[0,1]^d$ (possibly $\Lambda=[0,1]^d$), an {\em SLKKM coloring} of $\Lambda$ is a function $\chi:\Lambda\to C$ for some set $C$ such that $\chi$ does not map points on opposite faces to the same color. If $\abs{C}<\infty$, we call $\chi$ a {\em finite SLKKM coloring}.
\end{definition}


\begin{theorem-known}[KKM-Lebesgue Theorem]\label{kkm-lebesgue}
    Given a finite SLKKM coloring $\chi$ of $[0,1]^d$, there exists a point $\vec{p}\in[0,1]^d$ belonging to the closure of at least $d+1$ color sets (i.e. $\abs*[\big]{\set{c\in C:\vec{p}\in \clos{\chi^{-1}(c)}}}\geq d+1$).
\end{theorem-known}


The main result of this paper is that if we are interested in a small cubical region (an open $\ell_\infty$ $\epsilon$-ball), rather than a single point as in the \nameref{kkm-lebesgue}, then for any SLKKM coloring we can find a point $\vec{p}$ where the open $\ell_\infty$ $\epsilon$-ball at $\vec{p}$ intersects a significant number of colors. The focus on an open ball instead of a point is why we consider this to be a ``neighborhood'' variant of the cubical KKM lemma and the Lebesgue covering theorem. In the statement below, $\balloinf{\epsilon}{\vec{p}}$ denotes the open $\epsilon$-ball at $\vec{p}$ with respect to the $\ell_\infty$ norm.

\begin{restatable}[Neighborhood KKM-Lebesgue Theorem]{theorem}{restatableNeighborhoodVariant}\label{kkm-lebesgue-variant}
Given an SLKKM coloring of $[0,1]^d$, for any $\epsilon\in(0,\infty)$ there exists a point $\vec{p}\in[0,1]^d$ such that $\balloinf{\epsilon}{\vec{p}}$ contains points of at least $\ceil{\big(1+\frac{\epsilon}{1+\epsilon}\big)^d}$ different colors. In particular, if $\epsilon\in(0,\frac12]$ then $\balloinf{\epsilon}{\vec{p}}$ contains points of at least $\ceil{\big(1+\frac23 \epsilon\big)^d}$ different colors.
\end{restatable}

The proof will be given in \Autoref{sec:proof}, but for now we offer three comments about this theorem. First, we no longer need a finiteness assumption (as in the \nameref{kkm-lebesgue}) because we are not working with closures\footnote{
    The reason the \nameref{kkm-lebesgue} requires finiteness is because for a finite collection of sets, the union of closures equals the closure of unions, but this property does not hold in general for infinite collections. Even without the finiteness condition, it is still the case that there exists a point $\vec{p}$ such that every open set containing $\vec{p}$ intersects at least $d+1$ colors. See for example the proof of \Autoref{kkm-implies-kkm-lebesgue} where this is demonstrated. The \nameref{kkm-lebesgue} does not hold in general without the finiteness hypothesis as demonstrated by the non-finite SLKKM coloring which assigns a unique color to each point of $[0,1]^d$.
}.
Second, restricting to $\epsilon\leq\frac12$ in the above statement is reasonable because for $\epsilon>\frac12$ the $\ell_\infty$ $\epsilon$-ball can be placed at the center of the unit cube and is then a superset of the cube and hence it contains all the points in the cube. Note that no two vertices of the cube can have the same color in an SLKKM coloring as they belong to opposite faces. This means that every SLKKM coloring uses at least $2^d$ colors (and there exist SLKKM colorings with only $2^d$ colors---e.g. color each of the $2^d$ orthants a distinct color). Thus, for $\epsilon>\frac12$, we don't need the bound of $\ceil{\big(1+\frac{\epsilon}{1+\epsilon}\big)^d}$ in the \nameref{kkm-lebesgue-variant} since we just demonstrated a better (tight) bound of $2^d$. For this reason, it is sensible to only consider $\epsilon\leq\frac12$ in the statement of the \nameref{kkm-lebesgue-variant} and consequently obtain the cleaner (and only moderately worse) bound of $\ceil{\big(1+\frac23 \epsilon\big)^d}$.

Third, while the statement is given for fixed $\epsilon$ and $d$, we have already encountered applications of this theorem where $\epsilon$ is a function of the dimension $d$ (see \cite{geometry_of_rounding, FOCS23_submission}), so we briefly note the asymptotic behavior of this bound. If $\epsilon\in O\left(\frac1{d}\right)$, then our lower bound on the number of colors given by the \nameref{kkm-lebesgue-variant} is an $O(1)$ bound\footnote{
    This is because $\lim_{d\to\infty}(1+\frac{c}{d})^d=e^c$.
} which is asymptotically worse than the value $d+1$ given in the \nameref{kkm-lebesgue}. However, if $\epsilon\in\omega\left(\frac{\ln(d)}{d}\right)$ then our bound is super-polynomial in the dimension\footnote{
    Let $k=(1+\frac23\epsilon)^d$. Because $\epsilon\in(0,\frac12]$ we have $\frac23\epsilon\in(0,\frac13]$. Using the inequality $\ln(1+x)\geq\frac x2$ for small enough $x$ (in particular for $x\in(0,\frac13]$) we have $\ln(k)=d\ln(1+\frac23\epsilon)\geq \frac13 d\epsilon$, so $\epsilon\leq\frac{3\ln(k)}{d}$. Thus, if the bound $k$ is polynomial in $d$, then $\epsilon\in O\left(\frac{\ln(d)}{d}\right)$, so if $\epsilon\in\omega\left(\frac{\ln(d)}{d}\right)$ the $k$ is not polynomial in $d$.
}. In particular, if $\epsilon\in\Theta(1)$, then our bound is asymptotically exponential in $d$.

\subsection{The Lebesgue Covering Theorem and the KKM Lemma on the Cube}
\label{subsec:kkm-lebesgue}

We will now briefly discuss the standard statements of the Lebesgue covering theorem and the cubical version of the KKM Lemma and how they relate to the \namefullref{kkm-lebesgue}. We can capture the hypotheses of the Lebesgue covering theorem and the cubical KKM lemma with the following two definitions. A Lebesgue cover is a finite closed cover of $[0,1]^d$ with no set containing points on opposite faces. Informally, a KKM cover is a finite closed cover of $[0,1]^d$ with the sets of the cover associated to the vertices of the cube; we require that any point in a face $F$ is covered by one of the sets corresponding to the vertices defining $F$.

\begin{definition-known}[Lebesgue Cover]
A {\em Lebesgue cover} of $[0,1]^d$ is an indexed family $\C=\set{C_{n}}_{n\in[N]}$ of closed subsets of $\R^d$ (for some $N\in\N$) which covers $[0,1]^d$ such that for each $n\in[N]$, $C_n$ does not contain points on opposite faces of the cube. 
\end{definition-known}

\begin{definition-known}[KKM Cover]
A {\em KKM cover} of $[0,1]^d$ is an indexed family $\C=\set{C_{\vec{v}}}_{\vec{v}\in\set{0,1}^d}$ of closed subsets of $\R^d$ such that for each face $F$ of the cube, $F\subseteq\bigcup_{\vec{v}\in F\cap\set{0,1}^d} C_{\vec{v}}$. (In particular, the cube itself is a face, so the cube is covered.)
\end{definition-known}

While the two notions of covers are different from each other\footnote{
    For example, consider the indexed family $\C=\set{C_{\vec{v}}}_{\vec{v}\in\set{0,1}^d}$ where $C_{\vec{v}}=[0,1]^d$ for each $\vec{v}$; this is a KKM cover, but not a Lebesgue cover. Conversely, a finite Lebesgue cover with cardinality exceeding $\abs{\set{0,1}^d}=2^d$ is not a KKM cover.
}, they are both sufficient to guarantee the same conclusion as in the \nameref{kkm-lebesgue}: at least $d+1$ sets in the cover meet at some point.

\begin{theorem-known}[Lebesgue Covering Theorem]\label{true-lebesgue-covering-theorem}
Given a Lebesgue cover of $[0,1]^d$, there exists a point $\vec{p}\in[0,1]^d$ belonging to at least $d+1$ sets in the cover (i.e. $\abs{\set{n\in[N]:\vec{p}\in C_{n}}}\geq d+1$).
\end{theorem-known}

\begin{theorem-known}[Cubical KKM Lemma]\label{true-kkm-lemma}
Given a KKM cover of $[0,1]^d$, there exists a point $\vec{p}\in[0,1]^d$ belonging to at least $d+1$ sets in the cover (i.e. $\abs{\set{\vec{v}\in\set{0,1}^d:\vec{p}\in C_{\vec{v}}}}\geq d+1$).
\end{theorem-known}

These results are well known and can be found for example in \cite{Kuhn66, wolsey_cubical_1977, komiya_simple_1994, van_der_laan_intersection_1999, de_loera_polytopal_2002} for the \nameref{true-kkm-lemma} and in \cite{lebesgue_sur_1911} or \cite[Theorem~IV~2]{dimension_theory_lebesgue_covering} for the \nameref{true-lebesgue-covering-theorem}. Though the \nameref{true-lebesgue-covering-theorem}, the \nameref{true-kkm-lemma}, and the \nameref{kkm-lebesgue} (our coloring combination of the two) are all naturally equivalent to each other (a careful proof of this folklore result is given in \Autoref{subsec:kkm-lebesgue-coloring-equivalence}), we chose to use the \nameref{kkm-lebesgue} to represent these two more famous results and motivate our main result (the \nameref{kkm-lebesgue-variant}) for the following reasons.

For our purposes, we don't want to work directly with a KKM cover because we need to deal with extensions along boundaries that are cumbersome to work with for a KKM cover since sets can intersect opposite faces. We also don't want to work with Lebesgue covers because the definition of a Lebesgue cover actually implies that every set in the cover (when viewed as a subset of $[0,1]^d$) has $\ell_\infty$ diameter strictly less than $1$, but we are okay with sets that have diameter $1$ as long as they don't contain points on opposite faces (i.e. no points in the set attain $\ell_\infty$ distance $1$ from each other). Also, we don't want to require the sets to be closed which is formally a requirement of both types of covers. 




Nonetheless, as one may expect, our main result (the \nameref{kkm-lebesgue-variant}) can be equivalently restated by replacing the SLKKM coloring hypothesis with either a Lebesgue cover
(\Autoref{corollary:neighborhood-lebesgue}) 
or a KKM cover 
(\Autoref{corollary:neighborhood-kkm}). 
The proofs that these corollaries follow from the \nameref{kkm-lebesgue-variant} are given after the statements below. As this direction of the implications is sufficient to prove all three results, we don't provide the proofs of the reverse implications of the equivalence, but the main ideas are identical to those used to prove the equivalence of the \nameref{true-kkm-lemma}, the \nameref{true-lebesgue-covering-theorem}, and the \nameref{kkm-lebesgue} (found in \Autoref{subsec:kkm-lebesgue-coloring-equivalence}).


\begin{corollary}[Neighborhood Lebesgue Theorem]\label{corollary:neighborhood-lebesgue}
Given a Lebesgue cover of $[0,1]^d$, for any $\epsilon\in(0,\infty)$ there exists a point $\vec{p}\in[0,1]^d$ such that $\balloinf{\epsilon}{\vec{p}}$ intersects at least $\ceil{\big(1+\frac{\epsilon}{1+\epsilon}\big)^d}$ sets in the cover. In particular, if $\epsilon\in(0,\frac12]$ then $\balloinf{\epsilon}{\vec{p}}$ intersects at least $\ceil{\big(1+\frac23 \epsilon\big)^d}$ sets in the cover.
\end{corollary}

\begin{corollary}[Neighborhood KKM Theorem]\label{corollary:neighborhood-kkm}
Given a KKM cover of $[0,1]^d$, for any $\epsilon\in(0,\infty)$ there exists a point $\vec{p}\in[0,1]^d$ such that $\balloinf{\epsilon}{\vec{p}}$ intersects at least $\ceil{\big(1+\frac{\epsilon}{1+\epsilon}\big)^d}$ sets in the cover. In particular, if $\epsilon\in(0,\frac12]$ then $\balloinf{\epsilon}{\vec{p}}$ intersects at least$\ceil{\big(1+\frac23 \epsilon\big)^d}$ sets in the cover.
\end{corollary}

\begin{proof}[Proof of \Autoref{corollary:neighborhood-lebesgue} from the \namefullref{kkm-lebesgue-variant}]
~\\
Let $N\in\N$ and $\C=\set{C_n}_{n\in[N]}$ be a Lebesgue cover of $[0,1]^d$. Because this is a cover, every point of $[0,1]^d$ belongs to some set, so define $\chi$ as follows:
\begin{align*}
    \chi &: [0,1]^d \to [N]\\
    \chi(\vec{x}) &= \min\set{n\in[N]: \vec{x}\in C_{n}}.
\end{align*}
This is trivially a finite SLKKM coloring of $[0,1]^d$ because the codomain of $\chi$ is finite and for $\vec{x}^{(0)}$ and $\vec{x}^{(1)}$ on opposite faces, there is no $n\in[N]$ for which both $\vec{x}^{(0)}\in C_n$ and $\vec{x}^{(1)}\in C_n$ and thus $\set{n\in[N]: \vec{x}^{(0)}\in C_{n}}$ and $\set{n\in[N]: \vec{x}^{(1)}\in C_{n}}$ are disjoint, so $\chi(\vec{x}^{(0)})\not=\chi(\vec{x}^{(1)})$.

Note that $\chi^{-1}(n)$ is the set of points of color $n$. By the \nameref{kkm-lebesgue-variant}, there exists $\vec{p}\in[0,1]^d$ such that 
\[
\abs{\set{n\in[N]:\balloinf{\epsilon}{\vec{p}}\cap\chi^{-1}(n)}}\geq  \left(1+\tfrac{\epsilon}{1+\epsilon}\right)^d.
\]
Fix such a $\vec{p}$ for the remainder of the proof. For each $n\in[N]$, observe that $\chi^{-1}(n)\subseteq C_n$ because for any $\vec{x}\in\chi^{-1}(n)$ we have $\chi(\vec{x})=n$, so by definition of $\chi$ we have $\vec{x}\in C_n$. The following subset containment then follows immediately:
\[
\set{n\in[N]:\balloinf{\epsilon}{\vec{p}}\cap\chi^{-1}(n)\not=\emptyset} \quad\subseteq\quad \set{n\in[N]:\balloinf{\epsilon}{\vec{p}}\cap C_{n}\not=\emptyset}.
\]
and since the former has cardinality at least $(1+\frac{\epsilon}{1+\epsilon})^d$, so does the latter, which proves the result. (The ``in particular'' part of the result is just an arithmetic fact which is demonstrated in the proof of the \namefullref{kkm-lebesgue-variant}.)
\end{proof}

\begin{proof}[Proof of \Autoref{corollary:neighborhood-kkm} from the \namefullref{kkm-lebesgue-variant}]
~\\
Let $\C=\set{C_{\vec{v}}}_{\vec{v}\in\set{0,1}^d}$ be a KKM cover of $[0,1]^d$. For each $\vec{x}\in[0,1]^d$, let $F_{\vec{x}}$ denote the smallest face of the cube containing $\vec{x}$ (i.e. $F_{\vec{x}}$ is the intersection of all faces containing $\vec{x}$, and it is well-known and easily verified that this intersection is also a face). By the defining property of a KKM cover, we have $F_{\vec{x}}\subseteq\bigcup_{\vec{v}\in F_{\vec{x}}\cap\set{0,1}^d}C_{\vec{v}}$, so in particular there exists some $\vec{v}\in F_{\vec{x}}\cap\set{0,1}^d$ with $\vec{x}\in C_{\vec{v}}$. Define the function $\chi$ as follows where $\min_{\mathrm{lex}}$ denotes the minimum element in a subset of $\set{0,1}^d$ under the lexicographic ordering:
\begin{align*}
    \chi &:[0,1]^d \to \set{0,1}^d\\
    \chi(\vec{x}) &= \min_{\mathrm{lex}}\set{\vec{v}\in F_{\vec{x}}\cap\set{0,1}^d: \vec{x}\in C_{\vec{v}}}
\end{align*}
We have already demonstrated that the the set in the definition is not empty, so $\chi$ is well-defined.

We claim that $\chi$ is a finite SLKKM coloring of $[0,1]^d$. The finiteness is trivial because the codomain of $\chi$ is finite, so we need only show it is an SLKKM coloring. Suppose $F^{(0)}$ and $F^{(1)}$ are opposite faces of the cube (i.e. there is some coordinate $j\in[d]$ such that $\pi_j(F^{(0)})=\set{0}$ and $\pi_j(F^{(1)})=\set{1}$) and let $\vec{x}^{(0)}\in F^{(0)}$ and $\vec{x}^{(1)}\in F^{(1)}$. Because $\pi_j(F^{(0)})\cap \pi_j(F^{(1)})=\emptyset$, it follows that $F^{(0)}\cap F^{(1)}=\emptyset$, so $F^{(0)}$ and $F^{(1)}$ are disjoint sets.

Now consider the faces $F_{\vec{x}^{(0)}}$ and $F_{\vec{x}^{(1)}}$. Because $\vec{x}^{(0)}\in F^{(0)}$ and $F_{\vec{x}^{(0)}}$ is by definition the intersection of all faces containing $\vec{x}^{(0)}$, we have $F_{\vec{x}^{(0)}}\subseteq F^{(0)}$ (and similarly replacing ``$0$'' with ``$1$'') so that also $F_{\vec{x}^{(0)}}$ and $F_{\vec{x}^{(1)}}$ are disjoint. By definition of $\chi$ we have $\chi(\vec{x}^{(0)})\in F_{\vec{x}^{(0)}}$ and $\chi(\vec{x}^{(1)})\in F_{\vec{x}^{(1)}}$ showing that $\chi(\vec{x}^{(0)})\not=\chi(\vec{x}^{(1)})$, so $\chi$ is an SLKKM coloring.

By the \nameref{kkm-lebesgue-variant}, there exists $\vec{p}\in[0,1]^d$ such that 
\[
\abs{\set{\vec{v}\in\set{0,1}^d:\balloinf{\epsilon}{\vec{p}}\cap\chi^{-1}(\vec{v})\not=\emptyset}}\geq \left(1+\tfrac{\epsilon}{1+\epsilon}\right)^d.
\]
Fix such a $\vec{p}$ for the remainder of the proof. For each $\vec{v}\in\set{0,1}^d$, observe that $\chi^{-1}(\vec{v})\subseteq C_{\vec{v}}$ because for any $\vec{x}\in\chi^{-1}(\vec{v})$ we have $\chi(\vec{x})=\vec{v}$, so by definition of $\chi$ we have $\vec{x}\in C_{\vec{v}}$. The following subset containment then follows immediately:
\[
\set{\vec{v}\in\set{0,1}^d:\balloinf{\epsilon}{\vec{p}}\cap\chi^{-1}(\vec{v})\not=\emptyset} \quad\subseteq\quad \set{\vec{v}\in\set{0,1}^d:\balloinf{\epsilon}{\vec{p}}\cap C_{\vec{v}}\not=\emptyset}.
\]
and since the former has cardinality at least $(1+\frac{\epsilon}{1+\epsilon})^d$, so does the latter which proves the result. (The ``in particular'' part of the result is just an arithmetic fact which is demonstrated in the proof of the \namefullref{kkm-lebesgue-variant}.)
\end{proof}


\subsection{Sperner's Lemma on the Cube}

In light of the fact that our main result can (and should) be viewed as a neighborhood variant of the \nameref{true-lebesgue-covering-theorem} and the \nameref{true-kkm-lemma}, it is natural to ask whether we can obtain a discretized version of our result to serve as a neighborhood version of Sperner's lemma on the cube, and the answer is yes.

It is important to note that there are multiple natural ways to adapt Sperner's lemma from the simplex to the cube. For example, \cite{de_loera_polytopal_2002} considers subdividing the cube into simplices and using $2^d$ colors (because this is natural in the broader polytope setting in which they work) while \cite{Kuhn66} considers subdividing the cube into smaller cubes but still using $2^d$ colors, and \cite{wolsey_cubical_1977} considers a subdivision into smaller cubes with either $2^d$ colors or with $d+1$ colors. 

Clearly, it would not make sense to try to consider our style of result for the cubical adaptions of Sperner's lemma which use $d+1$ colors because we cannot possibly hope in general for some $\epsilon$-ball to intersect $\ceil{\big(1+\frac23\epsilon\big)^d}$-many colors if only $d+1$ colors are used. Thus, the natural context for us is the cubical adaptions of Sperner's lemma in which $2^d$ colors are used. 

Another observation is that arbitrary triangulations or cube decompositions of $[0,1]^d$ defined by some set of points $\Lambda$ are not strong enough for us. We are working with an $\epsilon$-ball, so we need to have some information about how close together points in $\Lambda$ are from each other. We define below the notion we will use to make our main result discrete; informally, a set $\Lambda$ is called $\rho$-proximate if every point within a face $F$ is distance at most $\rho$ from a point in $\Lambda$ which belongs to the same face $F$. In the formal definition below, while $\Lambda$ can be thought of as a discrete set, this is not necessary.

\begin{definition}[$\rho$-Proximate Set]\label{defn:proximate}
Let $\Lambda\subseteq[0,1]^d$. For $\rho\in[0,\infty)$, $\Lambda$ is called $\rho$-proximate if for every face $F$ of $[0,1]^d$ and for every $\vec{x}\in F$, there exists $\vec{y}\in F\cap\Lambda$ such that $\norm{\vec{x}-\vec{y}}_\infty\leq\rho$.
\end{definition}

\begin{remark}\label{proximate-rephrasing-remark}
We could have equivalently rephrased the last part of the definition as ``$\ldots$ for every face $F$ of $[0,1]^d$ and for every $\vec{x}\in F$, we have $F\cap\Lambda\cap\ballcinf{\rho}{\vec{x}}\not=\emptyset$''. 
Or we could have equivalently rephrased it as ``$\ldots$ for every face $F$ of $[0,1]^d$, we have $F\subseteq \bigcup_{\vec{p}\in F\cap\Lambda}\ballcinf{\rho}{\vec{p}}$''.
Or we could have equivalently rephrased it as ``$\ldots$ for every face $F$ of $[0,1]^d$, we have $F\subseteq (F\cap\Lambda)+\ballcinf{\rho}{\vec{0}}$'' where ``$+$'' denotes the Minkowski sum of the two sets.
\end{remark}

Of particular note is that for any $\rho\in[0,1]$ (taking $n=\floor{1/\rho}$) the grid $\set{0,\rho,2\rho,\ldots,(n-1)\rho,n\rho,1}^d$ of $(n+2)^d$ points is $\rho$-proximate. This grid occurs naturally in many contexts and is one of the simplest settings in which we can consider discretizing the main result.

Now, if we discretize the coloring hypothesis of the \nameref{kkm-lebesgue-variant} to a set of points that are $\rho$-proximate, then we get the same conclusion as the \nameref{kkm-lebesgue-variant} but have to account for this spacing using the triangle inequality. Therefore, an ``effective $\epsilon$'' of ``$\epsilon-\rho$'' appears in the result.

\begin{restatable}[Neighborhood Sperner Lemma]{theorem}{restatableNeighborhoodSperner}\label{sperner-variant}
Let $\rho\in[0,\infty)$ and $\Lambda\subseteq [0,1]^d$ be $\rho$-proximate. Let $\rho'=\min(\rho,\frac12)$. Given an SLKKM coloring of $\Lambda$, for any $\epsilon\in(0,\infty)$ there exists a point $\vec{p}\in[0,1]^d$ such that $\balloinf{\epsilon}{\vec{p}}$ contains points of at least $\ceil{\Big(1+\frac{(\epsilon-\rho')}{1+(\epsilon-\rho')}\Big)^d}$ different colors. In particular, if $\epsilon\in(0,\frac12]$ then $\balloinf{\epsilon}{\vec{p}}$ contains points of at least $\ceil{\big(1+\frac23 (\epsilon-\rho')\big)^d}$ different colors.
\end{restatable}

The proof is given in \Autoref{sec:sperner}.

\begin{remark}\label{equivalence-of-neighborhood-sperner-kkm-lebesgue}
The \nameref{sperner-variant} and the \nameref{kkm-lebesgue-variant} are in fact naturally equivalent. We will prove the \nameref{sperner-variant} as a corollary of the \nameref{kkm-lebesgue-variant} showing one direction of the equivalence. For the other direction, we recover the statement of the \nameref{kkm-lebesgue-variant} from the \nameref{sperner-variant} in the special case that $\Lambda=[0,1]^d$ and $\rho=0$ (because $[0,1]^d$ is $0$-proximate).
\end{remark}

It should be unsurprising that the definition of $\rho$-proximate includes the property that every point of the cube is within distance $\rho$ of a point in $\Lambda$ because we want to show that some point $\vec{p}$ is $\epsilon$-close to many colors, so we need to know that the color classes (and thus the points in $\Lambda$) are not all mutually far apart\footnote{
    A set with this requirement is called a $\rho$-net in some contexts.
}. The condition that such points must belong to the same face may be less obvious but probably not surprising considering the nature of Sperner's lemma and the KKM lemma; the reason we need such a property is demonstrated by the set $\Lambda=(0,1)^d$. If we didn't require that the point $\vec{y}$ in \Autoref{defn:proximate} is in both $F$ and $\Lambda$, then the set $\Lambda=(0,1)^d$ would be $\rho$-proximate for each $\rho\in(0,\infty)$, and yet it would be a valid SLKKM coloring to assign every point of $\Lambda$ the same color, and thus we could not guarantee the existence of an $\epsilon$-ball intersecting more than $1$ color. As a consequence of this requirement in \Autoref{defn:proximate}, we obtain the following simple property.

\begin{remark}\label{vertices-in-proximate-set-remark}
If $\Lambda$ is $\rho$-proximate for some $\rho\in[0,\infty)$, then $\Lambda$ contains all vertices of the cube (i.e. $\Lambda\supseteq\set{0,1}^d$). This is because for any vertex $\vec{v}\in\set{0,1}^d$, the singleton set $F=\prod_{i=1}^d\set{v_i}=\set{\vec{v}}$ is a face of the cube, so by definition of $\rho$-proximate, for each $\vec{x}\in F$, there is some associated $\vec{y}\in F\cap\Lambda$, so in particular $F\cap\Lambda\not=\emptyset$. Since $F$ contains only the point $\vec{v}$, we must have $\vec{v}\in\Lambda$.
\end{remark}



\subsection{Outline}

The remainder of the paper is laid out as follows. In \Autoref{sec:notation}, we briefly state additional notation we use. In \Autoref{sec:proof}, we prove the \nameref{kkm-lebesgue-variant}. In \Autoref{sec:sperner}, we prove the \nameref{sperner-variant} as a corollary. We finish in \Autoref{sec:discussion} with a short discussion of how we hope the bounds might be improved, and we mention some limitations on what is possible. Proofs of lemmas omitted from the main text are included in \Autoref{sec:appendix}. For completeness, we include a proof of the folklore \namefullref{lower-bound-cover-number-Rd} in \Autoref{appendix:measure-theory} which is stated later and is central to the proof of the \nameref{kkm-lebesgue-variant}.

\subsection{Notation}
\label{sec:notation}

We write $\ballostd{\epsilon}{\vec{p}}$ (resp. $\balloinf{\epsilon}{\vec{p}}$) to indicate the open ball of radius $\epsilon$ at $\vec{p}$ with respect to a generic norm $\norm{\cdot}$ (resp. the $\ell_\infty$ norm). Similarly, we use
$\ballcstd{\epsilon}{\vec{p}}$ and $\ballcinf{\epsilon}{\vec{p}}$ for closed balls.
We use $\unitballmeasurestd$
to denote the volume of the unit radius ball in dimension $d$ with respect to a generic norm $\norm{\cdot}$.
We write $m$ for the Lebesgue measure and $m_{out}$ for the induced outer measure. For two sets $A,B\subseteq\R^d$, we write $A+B$ to indicate the Minkowski sum $A+B=\set*{\vec{a}+\vec{b}:\vec{a}\in A\text{ and }\vec{b}\in B}$. 

\section{Proof of the \nameref*{kkm-lebesgue-variant}}
\label{sec:proof}

We will require a few lemmas for the proof which we state here. Proofs are provided in \Autoref{appendix-sub:brunn-minkowski}. The first result will later allow us to pass a result through a limit because the answer will be an integer.

\begin{restatable*}{fact}{restatableSmallerCeiling}\label{smaller-ceiling}
For any $\alpha\in\R$, there exists $\gamma\in\R$ such that $\gamma<\alpha$ and $\ceil{\gamma}=\ceil{\alpha}$.
\end{restatable*}

The second result is an inequality that can be proved using basic calculus. The interpretation of the result is that (for appropriate parameters) in the expression $(x^{1/d}+\alpha)^d$ we can ``approximately factor out'' the $x$ term because there is both a $d$th root and $d$th power involved to get the no larger expression $x(1+\alpha)^d$.
\begin{restatable*}{lemma}{restatableBrunnMinkowskiBoundLemma}\label{brunn-minkowski-bound-lemma}
For $d\in[1,\infty)$ and $x\in[0,1]$ and $\alpha\in[0,\infty)$, it holds that $(x^{1/d}+\alpha)^d\geq x(1+\alpha)^d$.
\end{restatable*}

The third result is a consequence of the \nameref{generalized-brunn-minkowski-inequality} (discussed in \Autoref{appendix-sub:brunn-minkowski}) which says that for non-empty measurable sets $A,B\subseteq\R^d$ for which $A+B$ is also measurable, it holds that $m(A+B)\geq\big(m(A)^{\frac1d}+m(B)^{\frac1d}\big)^d$. The main content of our corollary is dealing with non-measurable sets in such a way that we can apply the \nameref{generalized-brunn-minkowski-inequality}.
\begin{restatable*}{corollary}{restatableOuterMeasureBrunnMinkowskiBoundLInfty}\label{outer-measure-brunn-minkowsi-bound-l-infty}
Let $d\in\N$ and $Y\subseteq\R^d$ and $\epsilon\in(0,\infty)$. Then $Y+\balloinf{\epsilon}{\vec{0}}$ is open (and thus Borel measurable), and $m(Y+\balloinf{\epsilon}{\vec{0}})\geq \left(m_{out}(Y)^\frac1d+2\epsilon\right)^d$.
\end{restatable*}

The final result we need is heavily based on the ideas behind the common proof of Blitchfeldt's theorem and can be viewed as a continuous version of the pigeonhole principle with multiplicity. It says that if we have a family of subsets of $S$ and in total their measure is at least $k\cdot m(S)$, then there is a point belonging to at least $\ceil{k}$ sets in the family. While this is a known result, it is frequently enough stated informally or only in the (very simple) special case $\ceil{k}=2$ that for completeness and convenience we include a proof in \Autoref{appendix:measure-theory}.

\begin{restatable*}[Continuous Multi-Pigeonhole Principle]{proposition-known}{restatableLowerBoundCoverNumberRd}\label{lower-bound-cover-number-Rd}
  Let $d\in\N$ and $S\subset\R^d$ be measurable with finite measure. Let $\A$ be a family of measurable subsets of $S$ and let $k=\ceil{\frac{\sum_{A\in\A}m(A)}{m(S)}}$. If $k<\infty$, then there exists $\vec{p}\in S$ such that $\vec{p}$ belongs to at least $k$ members of $\A$. If $k=\infty$, then for any integer $n$, then there exists $\vec{p}\in S$ such that $\vec{p}$ belongs to at least $n$ members of $\A$.
\end{restatable*}

Now we are ready to prove the main result of this work restated below. The proof is illustrated in \Autoref{fig:sperner-kkm-coloring}.

\begin{figure}
\sbox1{%
    \begin{subfigure}{.3\textwidth}
        \includegraphics[width=\textwidth]{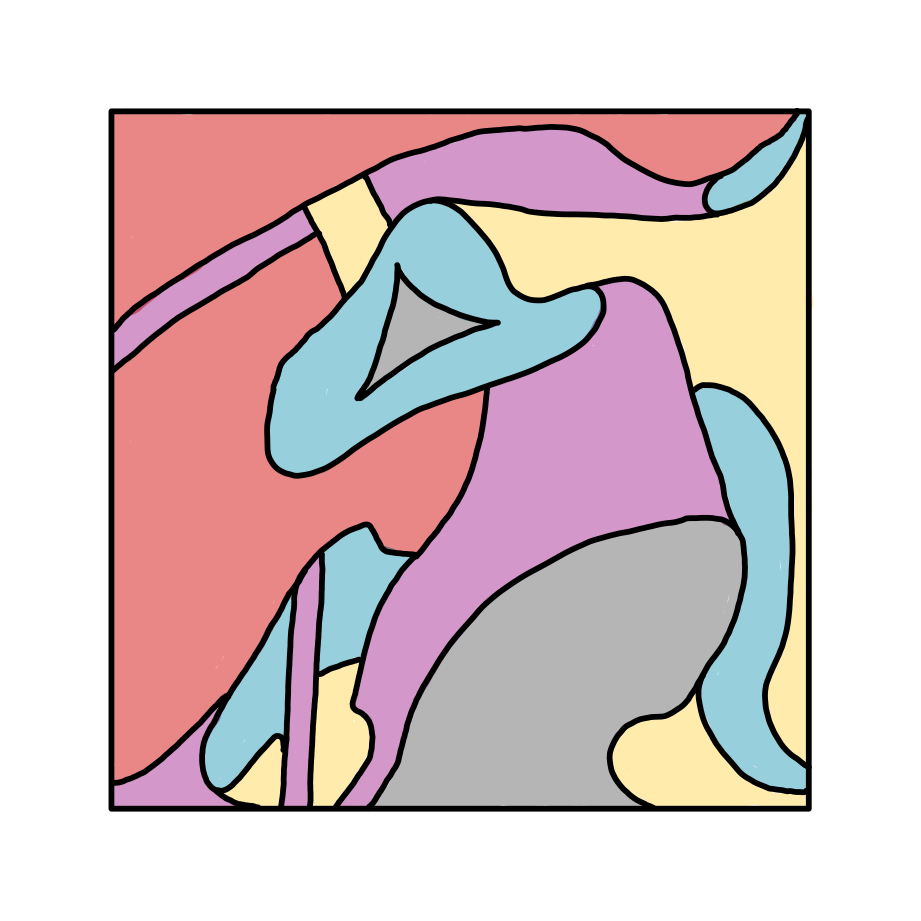}
        \subcaption{Initial coloring $\chi$}
        \label{subfig:1_non_spanning_coloring}
    \end{subfigure}
}
\sbox2{%
    \begin{subfigure}{.3\textwidth}
        \includegraphics[width=\textwidth]{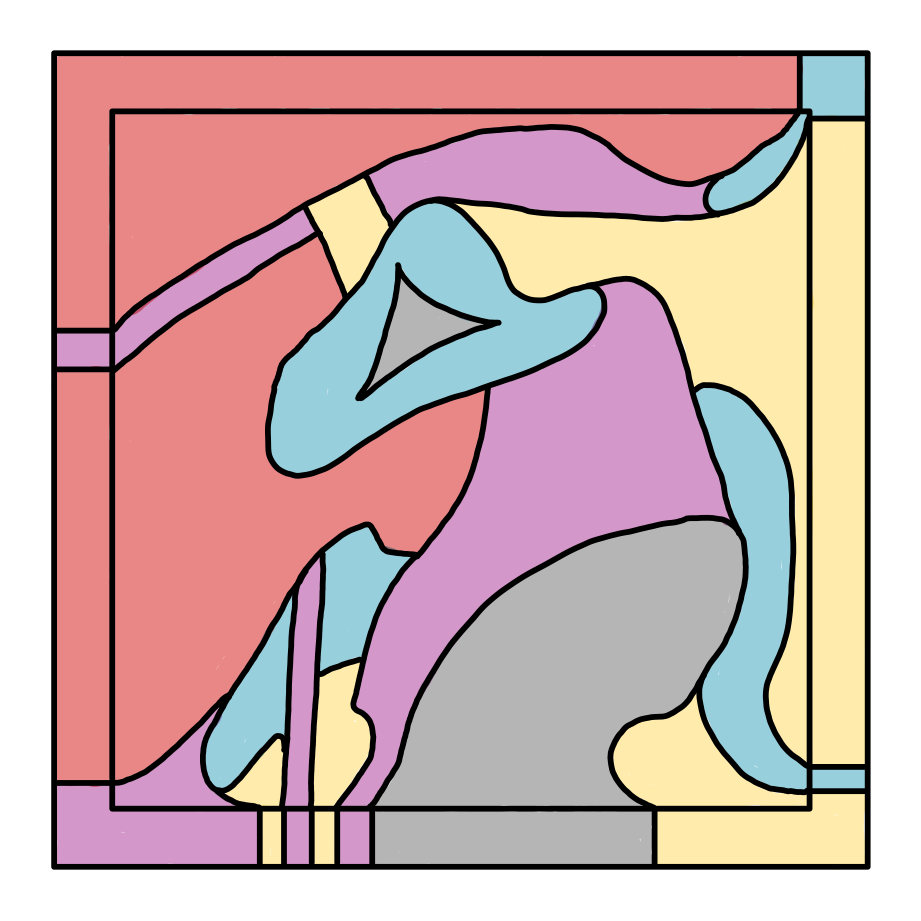}
        \subcaption{Extended coloring $\gamma$}
        \label{subfig:2_extended_coloring}
    \end{subfigure}
}
\sbox3{
    \begin{subfigure}{.3\textwidth}
        \includegraphics[width=\textwidth]{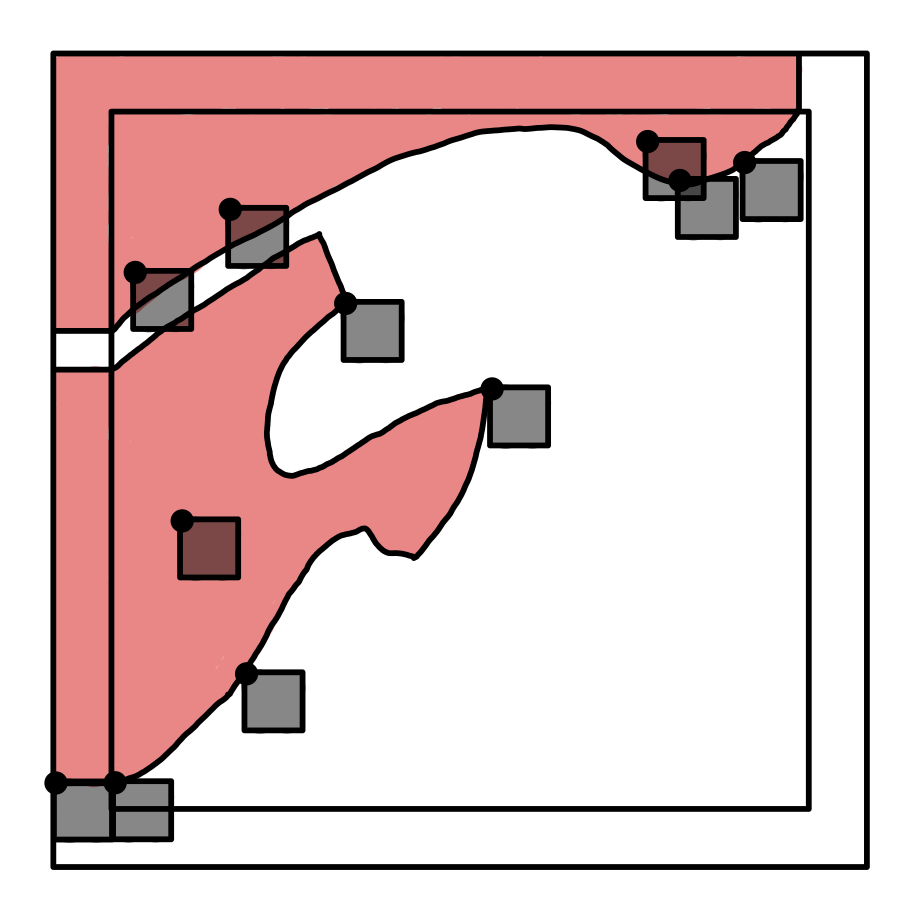}
        \subcaption{Red points ($Y_{\text{red}}$)}
        \label{subfig:3_red}
    \end{subfigure}
}
\sbox4{
    \begin{subfigure}{.3\textwidth}
        \includegraphics[width=\textwidth]{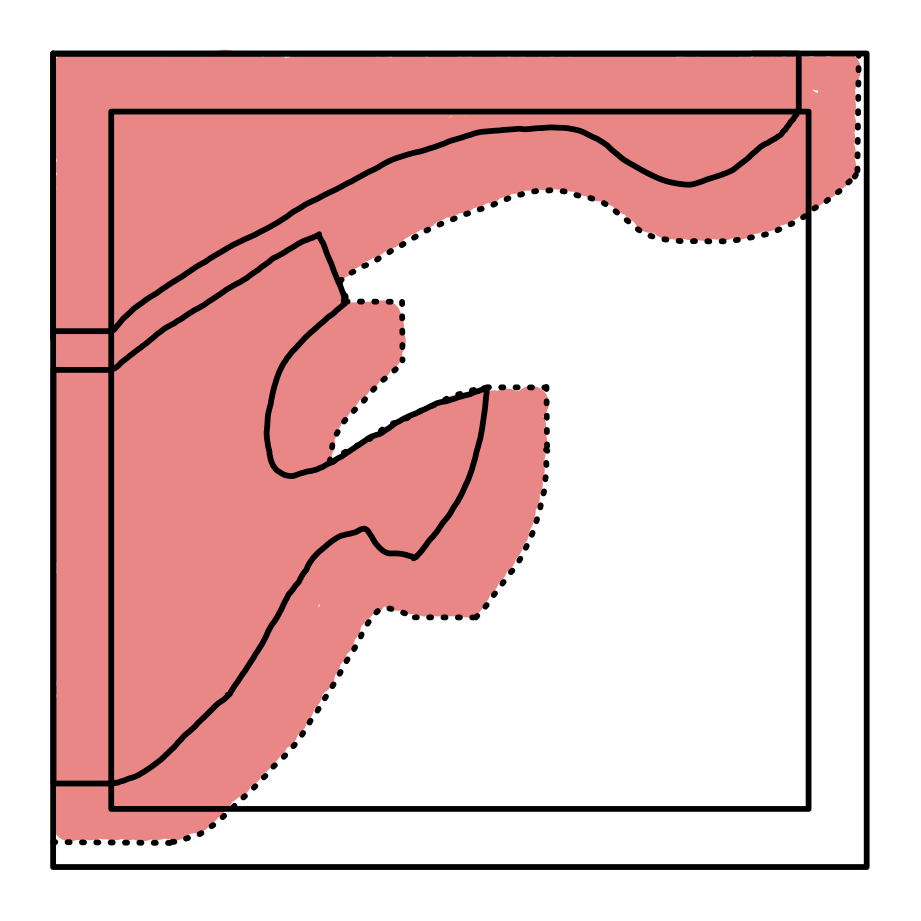}
        \subcaption{Ball added ($Y_{\text{red}}+B_{\vec{v}}$)}
        \label{subfig:4_red}
    \end{subfigure}
}
\sbox5{
    \begin{subfigure}{.3\textwidth}
        \includegraphics[width=\textwidth]{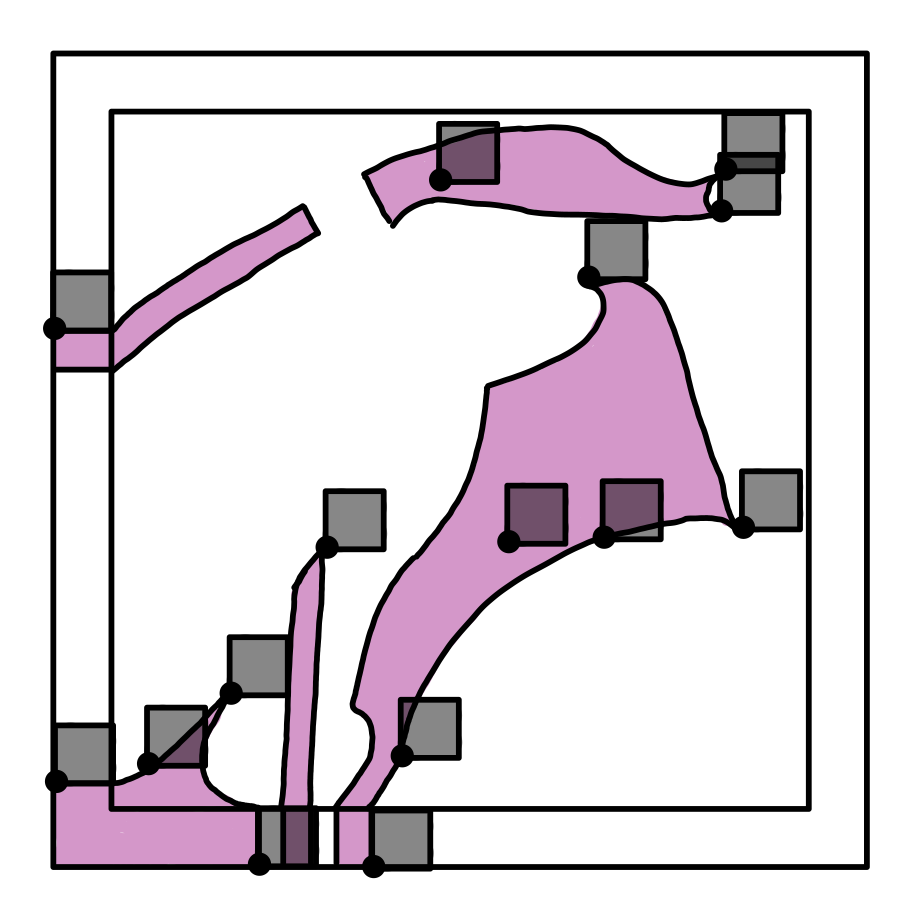}
        \subcaption{Purple points ($Y_{\text{purple}}$)}
        \label{subfig:5_purple}
    \end{subfigure}
}
\sbox6{
    \begin{subfigure}{.3\textwidth}
        \includegraphics[width=\textwidth]{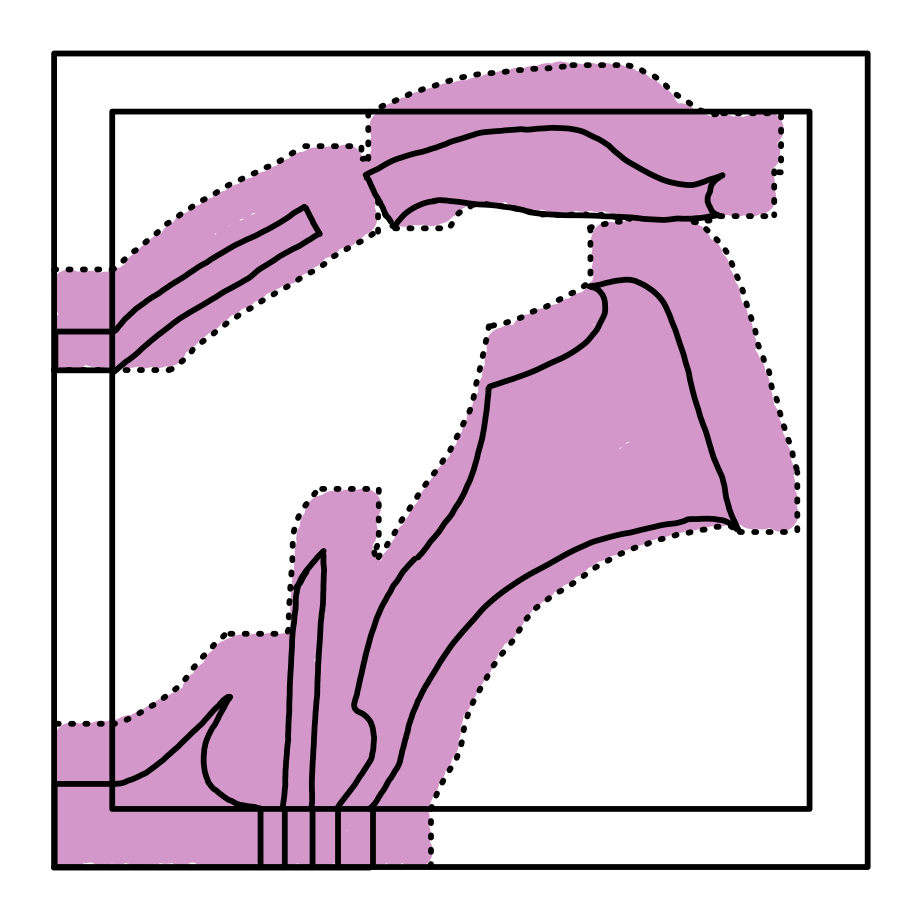}
        \subcaption{Ball added ($Y_{\text{purple}}+B_{\vec{v}}$)}
        \label{subfig:6_purple}
    \end{subfigure}
}
\sbox7{
    \begin{subfigure}{.3\textwidth}
        \includegraphics[width=\textwidth]{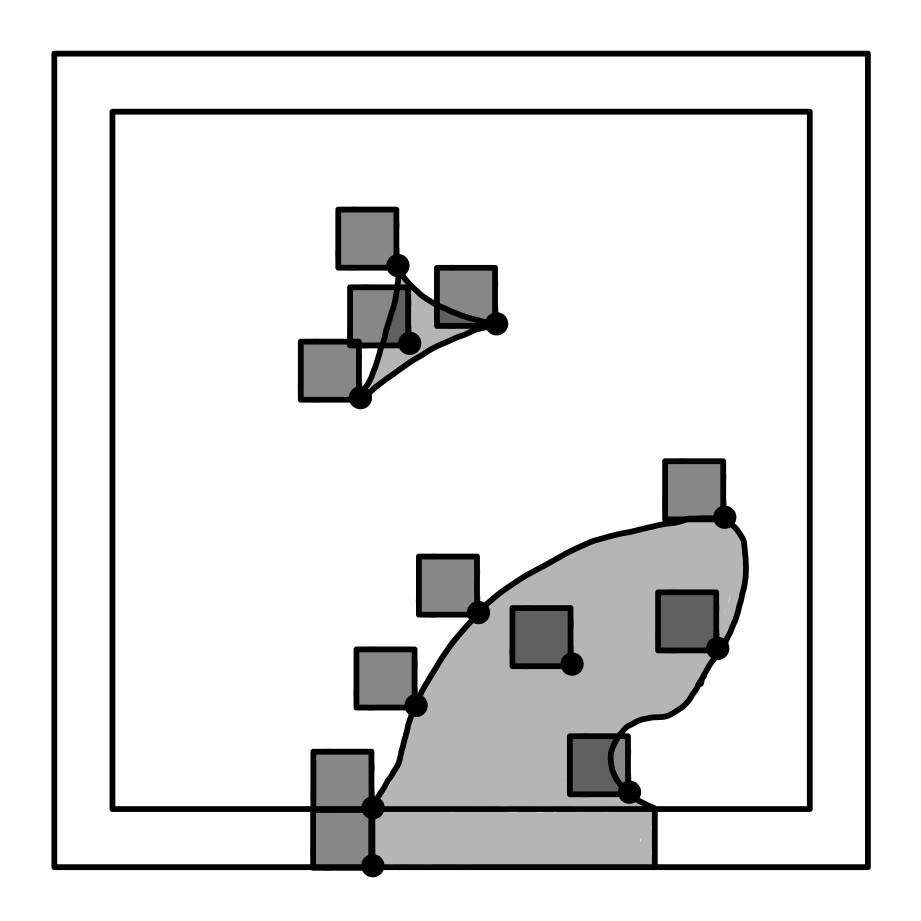}
        \subcaption{Gray points ($Y_{\text{gray}}$)}
        \label{subfig:7_gray}
    \end{subfigure}
}
\sbox8{
    \begin{subfigure}{.3\textwidth}
        \includegraphics[width=\textwidth]{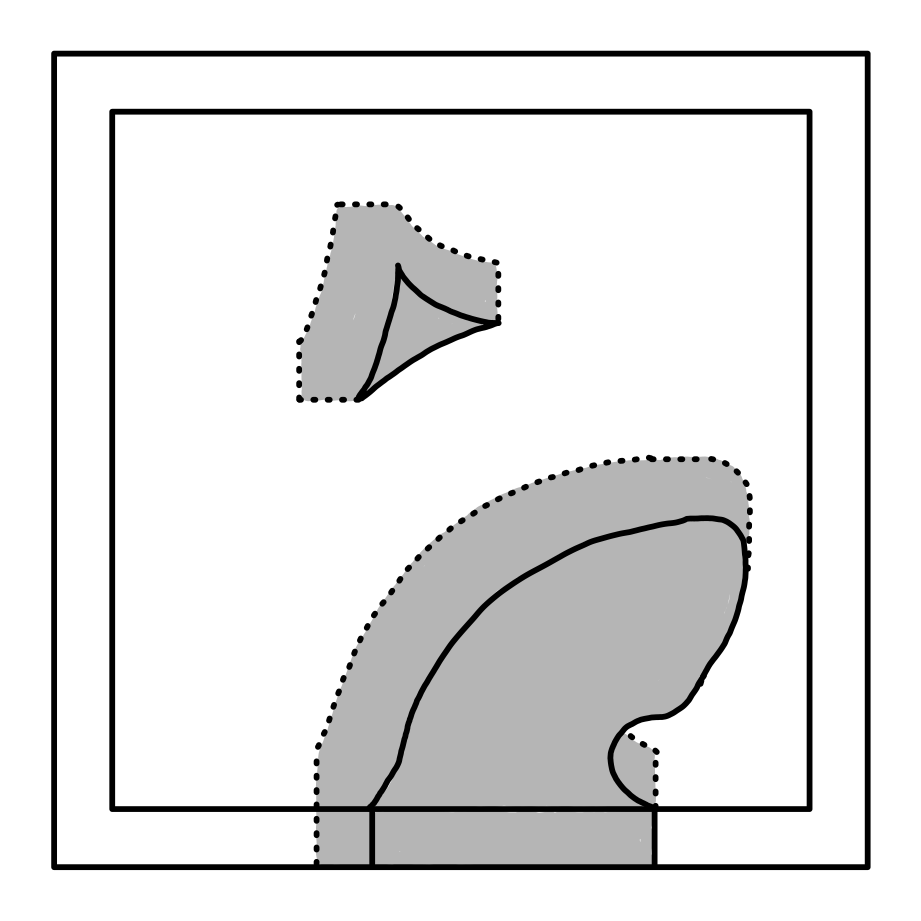}
        \subcaption{Ball added ($Y_{\text{gray}}+B_{\vec{v}}$)}
        \label{subfig:8_gray}
    \end{subfigure}
}
\sbox0{%
    \begin{subfigure}{.15\textwidth}
        \includegraphics[width=\textwidth]{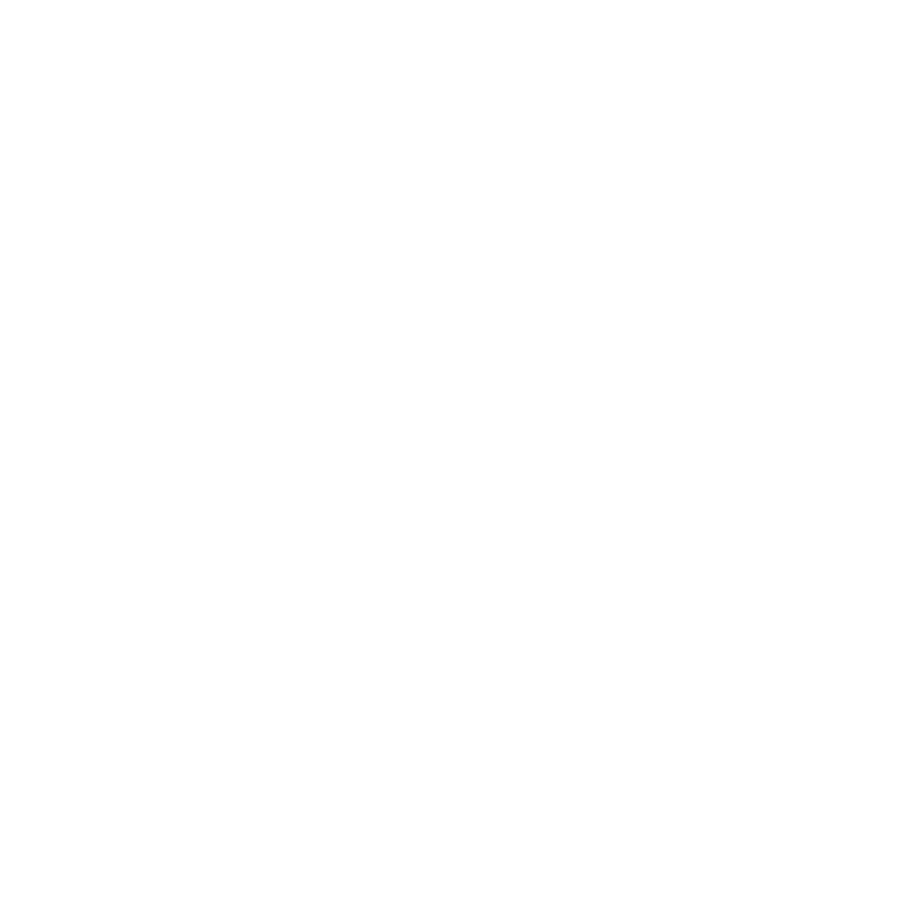}
    \end{subfigure}
}

\centering
\usebox0 \hfill \usebox1 \hfill \usebox2 \hfill \usebox0
\usebox3 \hfill \usebox5 \hfill \usebox7
\usebox4 \hfill \usebox6 \hfill \usebox8

\caption{
\subref{subfig:1_non_spanning_coloring} shows an SLKKM coloring $\chi$ of the unit cube $[-\frac12,\frac12]^d$ for $d=2$ (i.e. no color includes points on opposite edges).
\subref{subfig:2_extended_coloring} shows the natural extension $\gamma$ of that coloring to $[-\frac12-\epsilon,\frac12+\epsilon]^d$. The extension is obtained by mapping each point $\vec{y}\in[-\frac12-\epsilon,\frac12+\epsilon]^d$ to the point $\vec{x}\in[-\frac12,\frac12]^d$ for which each coordinate value is restricted to be within $[-\frac12,\frac12]$, and then $\vec{y}$ is given whatever color $\vec{x}$ had. \subref{subfig:3_red}, \subref{subfig:5_purple}, and \subref{subfig:7_gray} show three of the five colors and demonstrate that there is at least one quadrant of the $\epsilon$-ball that can be Minkowski summed with the color so that the sum remains a subset of the extended cube. For red it is the lower right quadrant, for purple it is the upper right, and for gray it could be the upper left (shown) or the upper right. \subref{subfig:4_red}, \subref{subfig:6_purple}, and \subref{subfig:8_gray} show the resulting Minkowski sum for each color. Utilizing the Brunn-Minkowski inequality, this set will have substantially greater area---by a factor of at least $(1+\frac{\epsilon}{1+\epsilon})^d$.
}

\label{fig:sperner-kkm-coloring}

\end{figure}

\restatableNeighborhoodVariant*
\begin{proof}
For convenience, we will assume that the cube is $[-\frac12,\frac12]^d$ rather than $[0,1]^d$. Let $C$ be a set (of colors) and $\chi\colon[-\frac12,\frac12]^d\to C$ be an SLKKM coloring of the unit cube $[-\frac12,\frac12]^d$. Let $C'=\range(\chi)$ so that we know every color in $C'$ appears for some point in the cube.

We first deal with the case where $C'$ has infinite cardinality\footnote{ 
    If one accepts the axiom of choice, then we don't need to deal with this as a special case, but by doing so, we can avoid requiring the axiom of choice in the proof. 
}. If $C'$ has infinite cardinality, then because we can cover the cube with finitely many $\epsilon$-balls, one of these balls must contain points of infinitely many colors, so the result is true. Thus, we assume from now on that $C'$ has finite cardinality.

For each color $c\in C'$ we will let $X_c$ denote the set of points assigned color $c$ by $\chi$---that is, $X_c=\chi^{-1}(c)$. Note that the hypothesis that no color includes points of opposite faces formally means that that for every color $c\in C'$, the set $X_c$ has the property that for each coordinate $i\in[d]$, the projection $\pi_i(X_c)=\set{x_i:\vec{x}\in X_c}$ does not contain both $-\frac12$ and $\frac12$.

The first step in the proof is to extend the coloring $\chi$ to the larger cube $[-\frac12-\epsilon,\frac12+\epsilon]^d$ in a natural way. Consider the following function $f$ which truncates points in the larger interval to be in the smaller interval:
\begin{align*}
&f\colon[-\tfrac12-\epsilon,\tfrac12+\epsilon]\to[-\tfrac12,\tfrac12] \\
&f(y)\defeq\begin{cases}
-\frac12 & y\leq-\frac12 \\
y        & y\in(-\frac12,\frac12) \\
\frac12 & y\geq\frac12
\end{cases}
\end{align*}
Let $\vec{f}\colon[-\tfrac12-\epsilon,\tfrac12+\epsilon]^d\to[-\tfrac12,\tfrac12]^d$ be the function which is $f$ in each coordinate: $\vec{f}(\vec{y})\defeq\langle f(y_i)\rangle_{i=1}^d$.

Now extend the coloring $\chi$ to the coloring $\gamma\colon[-\frac12-\epsilon,\frac12+\epsilon]^d\to C'$ defined by
\[
\gamma(\vec{x}) \defeq \chi\big(\vec{f}\left(\vec{x}\right)\big).
\]

For each color $c\in C'$, let $Y_c=\gamma^{-1}(c)$ denote the points assigned color $c$ by $\gamma$ and note that $X_{c}\subseteq Y_{c}$. Consistent with this notation, we will typically refer to a point in the unit cube as $\vec{x}$ and a point in the extended cube as $\vec{y}$.

We make the following claim which implies that for each color $c\in C'$, the set $Y_c$ of points of that color in the extended coloring are contained in a set bounded away from one side of the extended cube $[-\frac12-\epsilon,\frac12+\epsilon]^d$ in each coordinate.

\begin{subclaim}\label{kkm-lebesgue-variant-subclaim-1}
For each color $c\in C'$ there exists an orientation $\vec{v}\in\set{-1,1}^d$ such that $Y_c\subseteq\prod_{i=1}^d v_i\cdot(-\tfrac12,\tfrac12+\epsilon]$.
\end{subclaim}
\begin{subclaimproof}
Fix an arbitrary coordinate $i\in[d]$. Note that for every $\vec{y}\in Y_c$ we have $\vec{f}(\vec{y})\in X_c$ which is to say that the $\vec{y}$ has the same color in the extended coloring as $f(\vec{y})$ does in the original coloring (see justification\footnote{\label{footnote:matching-colors}
    For every $\vec{y}\in Y_c$ we have (by definition of $Y_c$) that $\gamma(\vec{y})=c$ and (by definition of $\gamma$) that $\gamma(\vec{y})=\chi(\vec{f}(\vec{y}))$ showing that $\chi(f(\vec{y}))=c$ and thus (by definition of $X_c$) that $f(\vec{y})\in X_c$. 
}).

Note that if there is some $\vec{y}\in Y_c$ with $y_i\leq-\frac12$, then $f(y_i)=-\frac12$ so the fact that $X_c\ni \vec{f}(\vec{y})$ implies that $\pi_i(X_c)\ni f(y_i)=-\frac12$. Similarly, if there is some $\vec{y}\in Y_c$ with $y_i\geq\frac12$, then $\pi_i(X_c)\ni\frac12$. Recall that by hypothesis, $\pi_i(X_c)$ does not contain both $-\frac12$ and $\frac12$ which means it is either the case that for all $\vec{y}\in Y_c$ we have $y_i>-\frac12$ (so $\pi_i(Y_c)\subseteq(-\frac12,\frac12+\epsilon]$) or it is the case that for all $\vec{y}\in Y_c$ we have $y_i<\frac12$ (so $\pi_i(Y_c)\subseteq[-\frac12-\epsilon,\frac12)$).

Thus we can choose $v_i\in\set{-1,1}$ such that $\pi_i(Y_c)\subseteq v_i\cdot(-\frac12,\frac12+\epsilon]$. Since this is true for each coordinate $i\in[d]$ we can select $\vec{v}\in\set{-1,1}^d$ such that
\[
Y_c \subseteq \prod_{i=1}^d\pi_i(Y_c) \subseteq \prod_{i=1}^d v_i\cdot(-\tfrac12,\tfrac12+\epsilon]
\]
as claimed.
\end{subclaimproof}

For an orientation $\vec{v}\in\set{-1,1}^d$, let $B_{\vec{c}}$ denote the set $B_{\vec{v}}\defeq\prod_{i=1}^d -v_i\cdot(0,\epsilon)$ which should be interpreted as a on open orthant of the $\ell_\infty$ $\epsilon$-ball centered at the origin---specifically the orthant opposite the orientation $\vec{v}$. Building on \Autoref{kkm-lebesgue-variant-subclaim-1}, we get the following:

\begin{subclaim}\label{subclaim-2}
For each color $c\in C'$, there exists an orientation $\vec{v}\in\set{-1,1}^d$ such that $Y_c+B_{\vec{v}} \subseteq [-\tfrac12-\epsilon,\tfrac12+\epsilon]^d$.
\end{subclaim}
\begin{subclaimproof}
Let $\vec{v}$ be an orientation given in \Autoref{kkm-lebesgue-variant-subclaim-1} for color $c$. We get the following chain of containments:
\begin{align*}
Y_c+B_{\vec{v}} &= Y_c + \left(\prod_{i=1}^d -v_i\cdot(0,\epsilon)\right) \tag{Def'n of $B_{\vec{v}}$} \\
&\subseteq \left( \prod_{i=1}^d v_i\cdot(-\tfrac12,\tfrac12+\epsilon] \right) + \left(\prod_{i=1}^d -v_i\cdot(0,\epsilon)\right) \tag{\Autoref{kkm-lebesgue-variant-subclaim-1}} \\
&= \left( \prod_{i=1}^d v_i\cdot(-\tfrac12,\tfrac12+\epsilon] \right) + \left(\prod_{i=1}^d v_i\cdot(-\epsilon,0)\right) \tag{Factor a negative} \\
&= \prod_{i=1}^d v_i\cdot(-\tfrac12-\epsilon,\tfrac12+\epsilon) \tag{Minkowski sum of rectangles} \\
%
&\subseteq [-\tfrac12-\epsilon,\tfrac12+\epsilon]^d \tag{$v_i\in\set{-1,1}$}.
\end{align*}
This proves the claim.
\end{subclaimproof}

We also claim that $Y_c+B_{\vec{v}}$ has substantial measure.

\begin{subclaim}\label{subclaim-3}
For each color $c\in C'$ and any orientation $\vec{v}\in\set{-1,1}^d$, the set $Y_c+B_{\vec{v}}$ is Borel measurable and $m(Y_c+B_{\vec{v}})\geq m_{out}(Y_c) \cdot \left(1+\frac{\epsilon}{1+\epsilon}\right)^d$.
\end{subclaim}
\begin{subclaimproof}
Let $M=(1+\epsilon)^d$ which is the measure of $\prod_{i=1}^d v_i\cdot(-\tfrac12,\tfrac12+\epsilon]$, and because by \Autoref{kkm-lebesgue-variant-subclaim-1}, $Y_c$ is a subset of some such set, we have $m_{out}(Y_c)\leq M$.

We have that $Y_c+B_{\vec{v}}$ is Borel measurable and that $m\left(Y_c+B_{\vec{v}}\right)\geq \left(m_{out}(Y_c)^\frac1d+\epsilon\right)^d$ by \Autoref{outer-measure-brunn-minkowsi-bound-l-infty} (because $B_{\vec{v}}$ is some translation of $\balloinf{\frac\epsilon2}{\vec{0}}$ and translations are irrelevant to the measure concerns of \Autoref{outer-measure-brunn-minkowsi-bound-l-infty}). Thus, we have the following chain of inequalities:
\begin{align*}
    m(Y_c+B_{\vec{v}}) &\geq \left(m_{out}(Y_c)^{1/d} + \epsilon\right)^d \tag{Above}\\
    &= M\cdot \left(\frac{m_{out}(Y_c)^{1/d}}{M^{1/d}} + \frac{\epsilon}{M^{1/d}}\right)^d \tag{Factor out $M$} \\
    &\geq M\cdot \left(\frac{m_{out}(Y_c)}{M}\right) \cdot \left(1+\frac{\epsilon}{M^{1/d}}\right)^d \tag{\Autoref{brunn-minkowski-bound-lemma}} \\
    &= m_{out}(Y_c) \cdot \left(1+\frac{\epsilon}{1+\epsilon}\right)^d \tag{Simplify and use $M=(1+\epsilon)^d$}
\end{align*}
\end{subclaimproof}

Now, consider the indexed family $\A=\set{Y_c+B_{\vec{v}(c)}}_{c\in C'}$ (where $\vec{v}(c)$ is an orientation for $c$ as in \Autoref{kkm-lebesgue-variant-subclaim-1} and \Autoref{subclaim-2}) noting that this family has finite cardinality because $C'$ has finite cardinality. Considering the sum of measures of sets in $\A$, we have the following:
\begin{align*}
    \sum_{A\in\A}m(A) &= \sum_{c\in C'}m\left(Y_c+B_{\vec{v}(c)}\right) \tag{Def'n of $\A$; measurability was shown above}\\
    &\geq \left(1+\frac{\epsilon}{1+\epsilon}\right)^d\cdot\sum_{c\in C'}m_{out}(Y_c) \tag{\Autoref{subclaim-3} and linearity of summation}\\
    &\geq \left(1+\frac{\epsilon}{1+\epsilon}\right)^d\cdot m_{out}\left(\bigcup_{c\in C'}Y_c\right) \tag{Countable/finite subaddativity of outer measures}\\
    &= \left(1+\frac{\epsilon}{1+\epsilon}\right)^d\cdot m_{out}\left([-\tfrac12-\epsilon,\tfrac12+\epsilon]^d\right) \tag{The $Y_c$'s partition $[-\tfrac12-\epsilon,\tfrac12+\epsilon]^d$}\\
    &= \left(1+\frac{\epsilon}{1+\epsilon}\right)^d\cdot (1+2\epsilon)^d \tag{Evaluate outer measure}\\
\end{align*}

By \Autoref{subclaim-2}, each member of $\A$ is a subset of $[-\frac12-\epsilon,\frac12+\epsilon]^d$, so by \Autoref{lower-bound-cover-number-Rd}, there exists a point $\vec{p}\in[-\tfrac12-\epsilon,\tfrac12+\epsilon]^d$ that belongs to at least
\[
\ceil{\frac{\left(1+\frac{\epsilon}{1+\epsilon}\right)^d\cdot (1+2\epsilon)^d}{(1+2\epsilon)^d}} = \ceil{\left(1+\frac{\epsilon}{1+\epsilon}\right)^d}
\]
sets in $\A$. That is, $\vec{p}$ belongs to $Y_c+B_{\vec{v}(c)}$ for at least $\ceil{\left(1+\frac{\epsilon}{1+\epsilon}\right)^d}$ colors $c\in C'$. For each such color $c$, it follows that $\vec{p}+(-\epsilon,\epsilon)^d$ intersects $Y_c$ (see justification\footnote{
    If $\vec{p}\in Y_c+B_{\vec{v}(c)} \subseteq Y_c + (-\epsilon,\epsilon)^d$, then by definition of Minkowski sum there exists $\vec{y}\in Y_c$ and $\vec{b}\in (-\epsilon,\epsilon)^d$ such that $\vec{p}=\vec{y}+\vec{b}$ so $Y_c\ni \vec{y}=\vec{p}-\vec{b}$ and also $\vec{p}-\vec{b}\in \vec{p}+(-\epsilon,\epsilon)^d$ demonstrating that these two sets contain a common point.
}). Note that $\vec{p}+(-\epsilon,\epsilon)^d=\balloinf{\epsilon}{\vec{p}}$ showing that $\balloinf{\epsilon}{\vec{p}}$ contains points of at least $\ceil{\left(1+\frac{\epsilon}{1+\epsilon}\right)^d}$ colors (according to the coloring of $\gamma$ since we are discussing sets $Y_c$).

What we really want, though, is a point in the unit cube that has this property rather than a point in the extended cube, and we want it with respect to the original coloring $\chi$ rather than the extended coloring $\gamma$. We will show that the point $\vec{f}\left(\vec{p}\right)$ suffices.

\begin{subclaim}\label{kkm-lebesgue-variant-subclaim-4}
If $c\in C'$ is a color for which $\balloinf{\epsilon}{\vec{p}}\cap Y_c \not=\emptyset$, then also $\balloinf{\epsilon}{\vec{f}\left(\vec{p}\right)}\cap X_c \not=\emptyset$.
\end{subclaim}
\begin{subclaimproof}
Let $\vec{y}\in \balloinf{\epsilon}{\vec{p}}\cap Y_{\vec{c}}$. Then because $\vec{y}\in \balloinf{\epsilon}{\vec{p}}$, we have $\norm{\vec{y}-\vec{p}}_\infty<\epsilon$, so for each coordinate $i\in[d]$ we have $\abs{y_i-p_i}<\epsilon$. It is easy to analyze the $9$ cases (or $3$ by symmetries) arising in the definition of $f$ to see that this implies $\abs{f(y_i)-f(p_i)}<\epsilon$ as well (i.e. $f$ maps pairs of values in its domain so that they are no farther apart), thus $\norm{\vec{f}\left(\vec{y}\right)-\vec{f}\left(\vec{p}\right)}_\infty<\epsilon$ and thus $\vec{f}\left(\vec{y}\right)\in \balloinf{\epsilon}{\vec{f}\left(\vec{p}\right)}$.

Also, as justified in a prior footnote\footref{footnote:matching-colors}, for any $\vec{y}\in Y_c$ we have $\vec{f}(\vec{y})\in X_c$ so that $\vec{f}\left(\vec{y}\right)\in\balloinf{\epsilon}{\vec{f}\left(\vec{p}\right)}\cap X_c$ which shows that the intersection is non-empty.
\end{subclaimproof}

Thus, because $\balloinf{\epsilon}{\vec{p}}$ intersects $Y_c$ for at least $\ceil{\left(1+\frac{\epsilon}{1+\epsilon}\right)^d}$ choices of color $c\in C'$, by \Autoref{kkm-lebesgue-variant-subclaim-4} $\vec{f}\left(\vec{p}\right)$ is a point in the unit cube for which $\balloinf{\epsilon}{\vec{f}\left(\vec{p}\right)}$ intersects $X_c$ for at least $\ceil{\left(1+\frac{\epsilon}{1+\epsilon}\right)^d}$ different colors $c\in C'$. That is, this ball contains points from at least this many of the original color sets.

The final step in the proof of the theorem is to clean up the expression with an inequality. Note that $C'$ must contain of at least $2^d$ colors because each of the $2^d$ corners of the unit cube must be assigned a unique color since any pair of corners belong to an opposite pair of faces on the cube. For this reason it is trivial that for $\epsilon>\frac12$ there is a point $\vec{p}$ such that $\balloinf{\epsilon}{\vec{p}}$ intersects at least $2^d$ colors: just let $\vec{p}$ be the midpoint of the unit cube. Thus, the only interesting case is $\epsilon\in(0,\frac12]$, and for such $\epsilon$ we have $1+\epsilon\leq\frac32$ and thus $\frac{\epsilon}{1+\epsilon}\geq\frac23\epsilon$ showing that $\left(1+\frac{\epsilon}{1+\epsilon}\right)^d\geq (1+\frac23\epsilon)^d$ which completes the proof.

\end{proof}
\section{Proof of the \nameref*{sperner-variant}}
\label{sec:sperner}

Now we restate and prove the \nameref{sperner-variant}. Basically, we assign each point of $[0,1]^d$ to a color of a point in $\Lambda$ nearby in a careful way and show that the resulting coloring still doesn't use the same color on opposite faces. We then use the \nameref{kkm-lebesgue-variant} to find the point $\vec{p}$ and pass back the result because $\Lambda$ approximates $[0,1]^d$. Part of the proof is illustrated in \Autoref{fig:proof-of-sperner-variant}.

\begin{figure}
\sbox1{%
    \begin{subfigure}{.3\textwidth}
        \includegraphics[width=\textwidth]{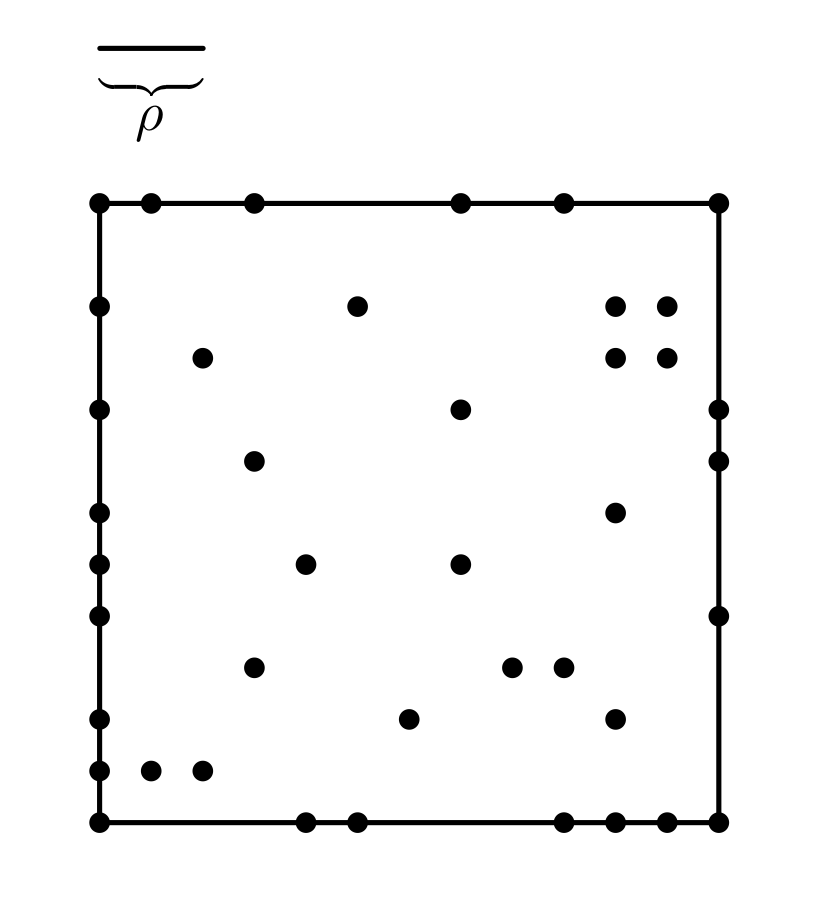}
        \subcaption{A $\rho$-proximate set $\Lambda$ with $\rho=\frac16$}
        \label{subfig:1}
    \end{subfigure}
}
\sbox2{%
    \begin{subfigure}{.3\textwidth}
        \includegraphics[width=\textwidth]{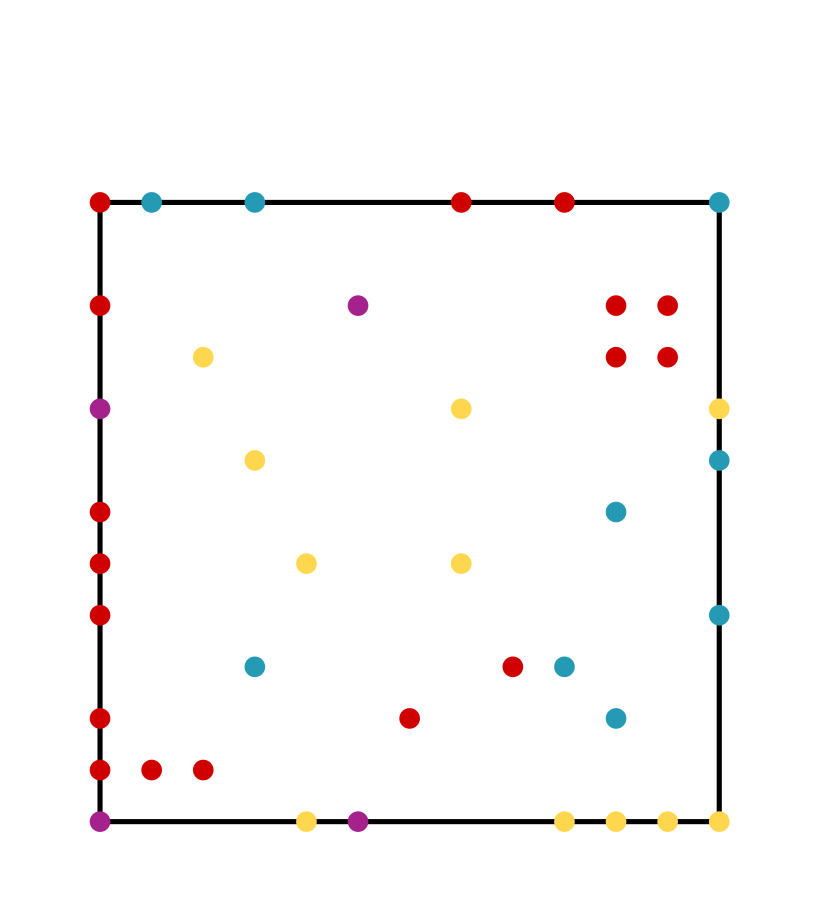}
        \subcaption{Coloring $\chi:\Lambda\to\set{R,B,P,Y}$}
        \label{subfig:2}
    \end{subfigure}
}
\sbox3{
    \begin{subfigure}{.3\textwidth}
        \includegraphics[width=\textwidth]{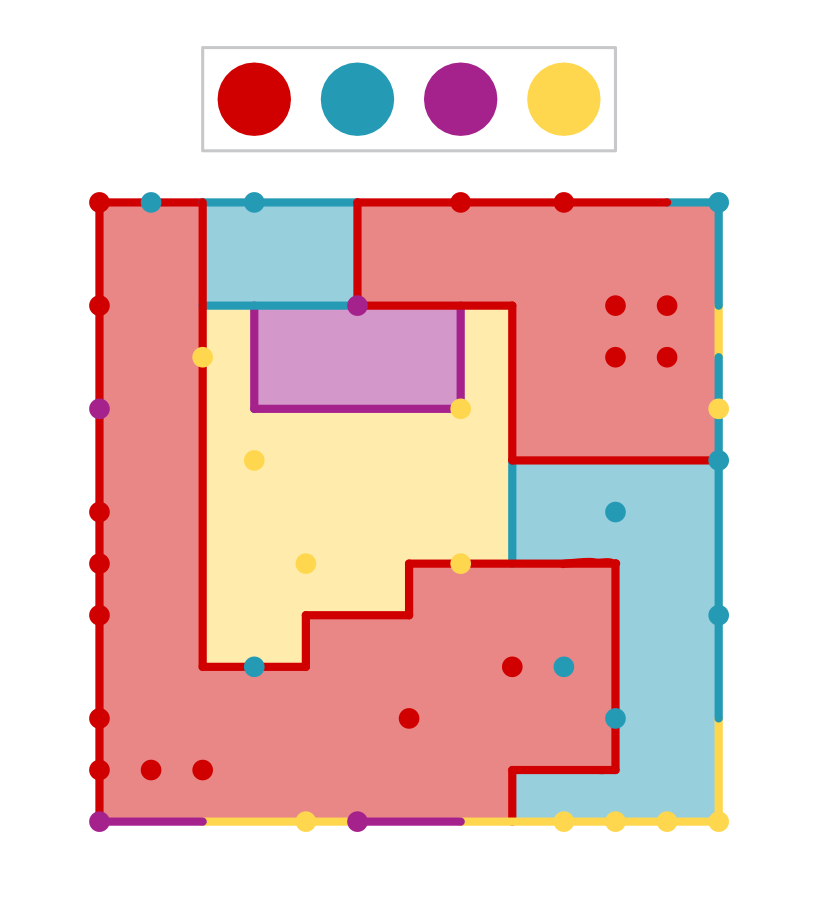}
        \subcaption{Sperner/KKM $\gamma:[0,1]^2\to\set{0,1}^2$}
        \label{subfig:3}
    \end{subfigure}
}
\sbox4{
    \begin{subfigure}{.3\textwidth}
        \includegraphics[width=\textwidth]{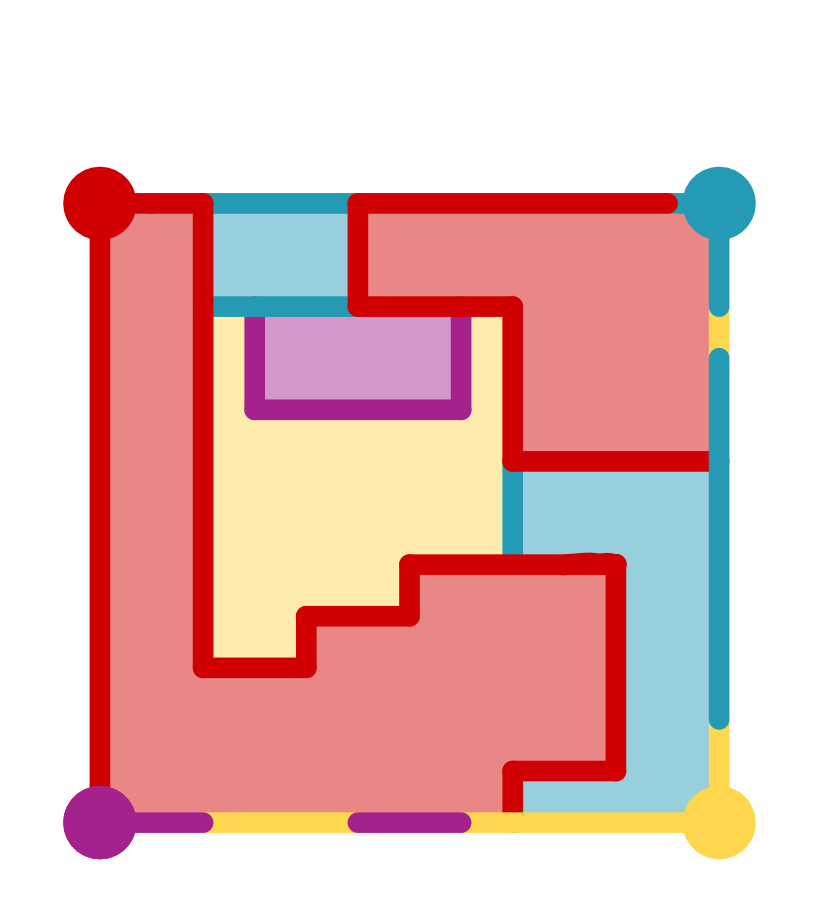}
        \subcaption{$\gamma$ emphasized}
        \label{subfig:4}
    \end{subfigure}
}
\sbox0{%
    \begin{subfigure}{.15\textwidth}
        \includegraphics[width=\textwidth]{images/kkm-lebesgue-coloring/blank_square.png}
    \end{subfigure}
}

\centering
\usebox0 \hfill \usebox1 \hfill \usebox2 \hfill \usebox0
\usebox0 \hfill \usebox3 \hfill \usebox4 \hfill \usebox0

\caption{
\subref{subfig:1} shows a set $\Lambda\subseteq[0,1]^2$ which is $\rho$-proximate for $\rho=\frac16$ (i.e. every vertex is displayed, every point of each edge is within distance $\frac16$ of some displayed point on the same edge, and every point in the interior is within distance $\frac16$ of some displayed point). 
The distance $\rho$ is shown visually in the upper left. \subref{subfig:2} shows an SLKKM coloring of $\Lambda$ (i.e. a function $\chi:\Lambda\to C=\{R_{ed},B_{lue},P_{urple},Y_{ellow}\}$ in which no color is used on opposite faces). 
\subref{subfig:3} shows how this coloring 
is used to produce a coloring of $[0,1]^2$ 
(i.e. a function $\gamma:[0,1]^d\to\set{0,1}^2$): an order is put on the set of colors $C$
(in this case $R_{ed},B_{lue},P_{urple},Y_{ellow}$ as shown at the top of \subref{subfig:3}) and each point $\vec{x}$ of the cube is mapped to the first color in the ordering for which there is a point $\vec{y}\in\Lambda$ of this color that is (1) within distance $\rho$ of $\vec{x}$, and (2) is in the smallest face containing $\vec{x}$.
For example, if $\vec{x}$ is on an edge, then the only points $\vec{y}$ considered are points in $\Lambda$ on the same edge which are distance at most $\rho$ away. \subref{subfig:4} clarifies the coloring $\gamma$ by emphasizing the colors on the vertices, the edges, and the boundaries between colors.
}

\label{fig:proof-of-sperner-variant}

\end{figure}

\restatableNeighborhoodSperner
\begin{proof}
Let $C$ be a set and $\chi:\Lambda\to C$ denote the coloring of $\Lambda$, and let $C'=\range(\chi)$.

We first deal with the case where $C'$ has infinite cardinality\footnote{ 
    If one accepts the axiom of choice, then we don't need to deal with this as a special case, but by doing so, we can avoid requiring the axiom of choice in the proof. 
}. If $C'$ has infinite cardinality, then because we can cover the cube (and thus $\Lambda$) with finitely many $\epsilon$-balls, one of these balls must contain points of infinitely many colors, so the result is true. Thus, we assume from now on that $C'$ has finite cardinality.

The next step in the proof is the following observation.
\begin{subclaim}
If $\rho\geq\frac12$ then the fact that $\Lambda$ is $\rho$-proximate implies that $\Lambda$ is $\frac12$-proximate.
\end{subclaim}
\begin{subclaimproof}
For any face $F$ of the cube, we have $F=\prod_{i=1}^d F_i$ where each $F_i$ is $\set{0}$, $\set{1}$, or $[0,1]$. This means that any $\vec{x}\in F$, we can find a vertex $\vec{v}$ of the cube (a point where each coordinate $v_i$ is $0$ or $1$) which also belongs to $F$ such that in each coordinate we have $\abs{x_i-v_i}\leq\frac12$. Since $\vec{v}$ is a vertex, by \Autoref{vertices-in-proximate-set-remark}, we have $\vec{v}\in\Lambda$. Thus, $\vec{v}\in F\cap\Lambda$ and $\norm{\vec{x}-\vec{v}}_\infty\leq\frac12$.
\end{subclaimproof}

Thus, we may continue knowing that $\Lambda$ is not only $\rho$-proximate but in fact $\rho'$-proximate where $\rho'=\min(\rho,\frac12)$. Next we comment on what happens if the term $(\epsilon-\rho')$ is non-positive. Essentially, we have to take show that the expressions $(1+\frac{(\epsilon-\rho')}{1+(\epsilon-\rho')})$ and $(1+\frac23(\epsilon-\rho'))$ are never large negative values to make sure the bound we are giving does not take on large positive values when $d$ is even. 

\begin{subclaim}
The stated result holds when $\epsilon\leq\rho'$.
\end{subclaim}
\begin{subclaimproof}
Note that when $\epsilon\leq\rho'$, then because $\rho'\in[0,\frac12]$ this implies $\epsilon\in(0,\frac12]$, so $(\epsilon-\rho')\in(-\frac12,0]$. On the interval $x\in(-\frac12,0]$, the expression $1+\frac1{1+x}$ is in $(0,1]$, so we have $1+\frac{(\epsilon-\rho')}{1+(\epsilon-\rho')}\in(0,1]$ and thus $(1+\frac{(\epsilon-\rho')}{1+(\epsilon-\rho')})^d\in(0,1]$, so $\ceil{(1+\frac{(\epsilon-\rho')}{1+(\epsilon-\rho')})^d}=1$.
Similarly, because $(\epsilon-\rho')\in(-\frac12,0]$, we have $\frac23(\epsilon-\rho')\in(-\frac13,0]$, so $1+\frac23(\epsilon-\rho')\in(\frac23,1]$, so $(1+\frac23(\epsilon-\rho'))^d\in(0,1]$, and again $\ceil{(1+\frac23(\epsilon-\rho'))^d}=1$.
Because $\Lambda$ is non-empty (by \Autoref{vertices-in-proximate-set-remark}) it is trivial to find a point where the $\epsilon$ ball contains points of at least $1$ color showing that the result is true when $\epsilon\leq\rho'$.
\end{subclaimproof}

Thus, we assume from now on that $\epsilon>\rho'$ which implies $(\epsilon-\rho')\in(0,\infty)$ (the hypothesis we need on the ball radius to apply the \nameref{kkm-lebesgue-variant}).

Next, for each $\vec{x}\in[0,1]^d$, let $F^{(\vec{x})}$ denote the smallest face containing $\vec{x}$ (i.e. the intersection of all faces containing $\vec{x}$) and let $C^{(\vec{x})}$ denote the set of colors present in the face $F^{(\vec{x})}$ and within $\rho'$ of $\vec{x}$ (formally, $C^{(\vec{x})}=\set{\chi(\vec{y}):\vec{y}\in F^{(\vec{x})}\cap\Lambda\cap\ballcinf{\rho'}{\vec{x}}}$) noting that $C^{(\vec{x})}$ is non-empty by \Autoref{proximate-rephrasing-remark}.

Let $\gamma:[0,1]^d\to C'\subseteq C$ map each $\vec{x}$ to some\footnote{
    Because $C'$ has finite cardinality, this does not require the axiom of choice.
} color in $C^{(\vec{x})}$. We claim that $\gamma$ is an SLKKM coloring so that we will be able to apply the \nameref{kkm-lebesgue-variant} to $\gamma$.

\begin{subclaim}\label{sperner-variant-subclaim-non-spanning-coloring}
The coloring $\gamma$ is does not assign the same color to points on opposite faces.
\end{subclaim}
\begin{subclaimproof}
We show that points on opposite faces are assigned different colors by $\gamma$. Let $F^{(0)}$ and $F^{(1)}$ be opposite faces of the cube---that is, there is some coordinate $j$ such that $\pi_{j}(F^{(0)})=\set{0}$ and $\pi_{j}(F^{(1)})=\set{1}$.

Let $\vec{x}^{(0)}\in F^{(0)}$ be arbitrary. Because $F^{(\vec{x}^{(0)})}$ is the intersection of all faces containing $\vec{x}$, we have $\vec{x}^{(0)}\in F^{(\vec{x}^{(0)})}\subseteq F^{(0)}$. By definition of $\gamma$, there is some $\vec{y}^{(0)}\in F^{(\vec{x}^{(0)})}\cap\Lambda\subseteq F^{(0)}\cap\Lambda$ such that $\gamma(\vec{x}^{(0)})=\chi(\vec{y}^{(0)})$. Similarly, there is some $\vec{y}^{(1)}\in F^{(1)}\cap\Lambda$ such that $\gamma(\vec{x}^{(1)})=\chi(\vec{y}^{(1)})$.

By hypothesis of the coloring $\chi$, because $\vec{y}^{(0)}$ and $\vec{y}^{(1)}$ belong to opposite faces of the cube (i.e. $F^{(0)}$ and $F^{(1)}$), we have $\chi(\vec{y}^{(0)})\not=\chi(\vec{y}^{(0)})$ showing that $\gamma(\vec{x}^{(0)})\not=\gamma(\vec{x}^{(1)})$.
\end{subclaimproof}


The following claim says that for any point $\vec{p}$, all of the colors (of $\gamma$) appearing in the $(\epsilon-\rho')$ ball at $\vec{p}$ also appear (as colors of $\chi$) in $\Lambda$ within the $\epsilon$ ball at $\vec{p}$. The connection back to $\Lambda$ below is because for any $c\in C'$, we have $\chi^{-1}(c)\subseteq\Lambda$.

\begin{subclaim}\label{sperner-variant-subclaim-containment}
The following subset containment holds for each point $\vec{p}\in[0,1]^d$ (see comment\footnote{
    In the former set we express $c\in\range(\gamma)$ rather than $c\in C'$, because it is possible that $\range(\gamma)\subsetneq C'$. The coloring $\gamma$ was defined as one of many choices, and in the natural choices we would have $\gamma(\vec{y})=\chi(\vec{y})$ for each $\vec{y}\in\Lambda$, but we did not require this, so $\gamma$ is not necessarily an extension of $\chi$, and it is possible that it does not surject onto $C'$.
}):
\[
\set{c\in\range(\gamma):\gamma^{-1}(c)\cap\balloinf{\epsilon-\rho'}{\vec{p}}\not=\emptyset}
\quad\subseteq\quad
\set{c\in\range(\chi)=C':\chi^{-1}(c)\cap\balloinf{\epsilon}{\vec{p}}\not=\emptyset}
\]
\end{subclaim}
\begin{subclaimproof}
If $c$ belongs to the left set, then $\gamma^{-1}(c)\cap\balloinf{\epsilon-\rho'}{\vec{p}}\not=\emptyset$, so let $\vec{x}\in\gamma^{-1}(c)\cap\balloinf{\epsilon-\rho'}{\vec{p}}$. Then $\gamma(\vec{x})=c$ and using the defining property of $\gamma$ that $\gamma(\vec{x})\in C^{(\vec{x})}$, we have the following:
\[
c = \gamma(\vec{x}) \in C^{(\vec{x})} = \set{\chi(\vec{y}):\vec{y}\in F^{(\vec{x})}\cap\Lambda\cap\ballcinf{\rho'}{\vec{x}}}.
\]
This means that there is some $\vec{y}\in\Lambda\cap\ballcinf{\rho'}{\vec{x}}$ such that $\chi(\vec{y})=c$ (i.e. $\vec{y}\in\chi^{-1}(c)$). Also, for this $\vec{y}$, because $\norm{\vec{y}-\vec{x}}_\infty\leq\rho'$ and $\norm{\vec{x}-\vec{p}}_\infty<\epsilon-\rho'$ we have by the triangle inequality that $\vec{y}\in\balloinf{\epsilon}{\vec{p}}$. Thus, $\vec{y}$ demonstrates that $\chi^{-1}(c)\cap\balloinf{\epsilon}{\vec{p}}$ is non-empty which shows that $c$ belongs to the right set.
\end{subclaimproof}

Now we can finish off the proof. By the \nameref{kkm-lebesgue-variant} (using the fact that $(\epsilon-\rho')\in(0,\infty)$ along with \Autoref{sperner-variant-subclaim-non-spanning-coloring}), there exists $\vec{p}\in[0,1]^d$ such that $\balloinf{\epsilon-\rho'}{\vec{p}}$ intersects at least 
$\ceil{(1+\frac{(\epsilon-\rho')}{1+(\epsilon-\rho')})^d}$ colors (formally, $\set{c\in\range(\gamma): \gamma^{-1}(c)\cap \balloinf{\epsilon-\rho'}{\vec{p}}\not=\emptyset}$ has cardinality at least $\ceil{(1+\frac{(\epsilon-\rho')}{1+(\epsilon-\rho')})^d}$). Thus, by \Autoref{sperner-variant-subclaim-containment} the set $\set{c\in\range(\chi)=C': \chi^{-1}(c)\cap \balloinf{\epsilon}{\vec{p}}\not=\emptyset}$ also has cardinality at least $\ceil{(1+\frac{(\epsilon-\rho')}{1+(\epsilon-\rho')})^d}$ which is what we set out to prove. (Informally, this latter set is the colors $c$ for which there is a point in $\Lambda\cap\balloinf{\epsilon}{\vec{p}}$ that is mapped to $c$ by the original coloring $\chi$.)

Finally, if $\epsilon\in(0,\frac12]$, then because we have at this point that $\epsilon>\rho'$, we in fact have $\epsilon-\rho'\in(0,\frac12]$. Thus, by the same inequalities used in the proof of the \nameref{kkm-lebesgue-variant}, we have
$\ceil{\left(1+\frac{(\epsilon-\rho')}{1+(\epsilon-\rho')}\right)^d}\geq \ceil{(1+\tfrac23(\epsilon-\rho'))^d}$ which completes the proof.
\end{proof}
\section{Discussion}
\label{sec:discussion}

In this section, we will give some discussion of the bound that we achieved in the \nameref{kkm-lebesgue-variant} (and equivalently in the \nameref{sperner-variant}) including some limitations on improving that bound and some desires for future improvements of our result. We begin by defining the best possible function we could use to replace the expression ``$\ceil{(1+\frac\epsilon{1+\epsilon})^d}$'' in the statement of the \nameref{kkm-lebesgue-variant}. Let $K^\circ:\N\times(0,\infty)\to\N$ be defined by
\[
K^\circ(d,\epsilon)\defeq \max\Bigg\lbrace\kappa\in\N:\;\; \parbox{2.5in}{for every SLKKM coloring of $[0,1]^d$, there exists a point $\vec{p}\in[0,1]^d$ such that
$\balloinf{\epsilon}{\vec{p}}$ intersects at least $\kappa$ colors}\Bigg\rbrace
\]
and define the function $\clos{K}$ similarly but using closed balls instead of open balls.

A few comments are in order. First, this maximum is well defined because there exist SLKKM colorings of $[0,1]^d$ using $2^d$ colors (e.g. color each of the $2^d$ orthants of the cube a unique color dealing with boundaries between colors arbitrarily) and so the set that the maximum is being taken over is necessarily a subset of $[2^d]$ and thus has a maximum. Second, it is trivially the case for all $d\in\N$ and $\epsilon\in(0,\infty)$ that $\clos{K}(d,\epsilon)\geq K^\circ(d,\epsilon)$ because the closed $\epsilon$-ball is a superset of the open $\epsilon$-ball. Third, by the same superset reasoning we also know that for each fixed $d\in\N$, both functions $K^\circ$ and $\clos{K}$ are non-decreasing in $\epsilon$ (we will refer to both of these properties as monotonicity throughout the discussion). We can say much more about $K^\circ$ and $\clos{K}$, and the information from the discussion that follows is summarized in \Autoref{tab:K-bounds}.

\paragraph{Immediate lower bound:} The first obvious thing we can say is that for all $d\in\N$ and $\epsilon\in(0,\infty)$, we know that $K^\circ(d,\epsilon)\geq d+1$ (ditto for $\clos{K}$) which follows straightforwardly from the \nameref{kkm-lebesgue} taking care with the infinite cardinality\footnote{
    One can ``collapse'' the possibly infinitely many colors into just $2^d$ colors to obtain a new finite SLKKM coloring, and by the \nameref{kkm-lebesgue-variant}, there is a point at the closure of at least $d+1$ of the collapsed colors, and any open set around this point intersects at least $d+1$ of the original color sets. See \Autoref{kkm-implies-kkm-lebesgue} for example, where this result shows up as part of the larger proof.
}. This is a better bound for small $\epsilon$ (relative to $d$) than our bound in the \nameref{kkm-lebesgue-variant} because $\lim_{\epsilon\to0}\ceil{(1+\frac\epsilon{1+\epsilon})^d}=1<d+1$, so for sufficiently small $\epsilon$ it holds that $\ceil{(1+\frac\epsilon{1+\epsilon})^d}\leq d+1$.

\paragraph{Trivial tight bound for large $\boldsymbol{\epsilon}$:} The second obvious thing we can say is that for any $d\in\N$ and $\epsilon>\frac12$, we have $K^\circ(d,\epsilon)=2^d$. This is because the open $\ell_\infty$ $\epsilon$-ball placed at the center of the unit $d$-cube is a superset of the cube itself, so it intersects all $2^d$ corners---no two of which have the same color in an SLKKM coloring---so $K^\circ(d,\epsilon)\geq2^d$. And (by definition) $K^\circ(d,\epsilon)\leq2^d$ because there are SLKKM colorings with only $2^d$ colors. Similarly, for any $d\in\N$ and $\epsilon\geq\frac12$ (note the non-strict inequality this time), we have $\clos{K}(d,\epsilon)=2^d$.

\paragraph{Perspective of this paper:} The third thing that we can say is that for all $d\in\N$ and $\epsilon\in(0,\infty)$ we have $\clos{K}(d,\epsilon)\geq K^\circ(d,\epsilon)\geq \ceil{(1+\tfrac{\epsilon}{1+\epsilon})^d}$ by the \nameref{kkm-lebesgue-variant}. In other words, the purpose of this paper is to put some lower bound on $K^\circ$ (and consequently on $\clos{K}$).

\paragraph{Non-trivial tight bound for small $\boldsymbol{\epsilon}$:} The fourth thing we can say is something that we don't believe is at all obvious; it turns out that we know the value of $K^\circ$ (and $\clos{K}$) exactly for all dimensions for a specific small regime of $\epsilon$ near $0$. Specifically, for any $d\in\N$ and $\epsilon\in(0,\frac1{2d}]$, we know that $d+1\geq \clos{K}(d,\epsilon) \geq K^\circ(d,\epsilon)$; along with the ``first obvious thing'' we said, this gives equality with $d+1$. The reason for this is the following. We demonstrated in \cite[Theorem~7.20]{geometry_of_rounding} that in each dimension $d$, there is a partition $\P_{d}$ of $\R^d$ consisting of translates of the half-open cube $[0,1)^d$ with the property that for every point $\vec{p}\in\R^d$, the closed ball $\ballcinf{\frac1{2d}}{\vec{p}}$ intersects at most $d+1$ cubes in the partition (we called such partitions $(d+1,\frac{1}{2d})$-secluded partitions). This immediately gives an SLKKM coloring of $[0,1]^d$ with the same property: define the coloring $\chi:[0,1]^d\to\P_d$ by mapping each point in $[0,1]^d$ to the unique member of $\P_d$ which it belongs to. This is an SLKKM coloring because no member of $\P_d$ contains points $\ell_\infty$ distance $1$ apart, and so no points distance $1$ apart are given the same color (i.e. no points on opposite faces are given the same color). In fact, this coloring uses exactly $2^d$ colors\footnote{
    The range of $\chi$ can be shown to have cardinality exactly $2^d$.
    This is because for any $X\in\P_d$ which intersects $[0,1]^d$, it follows by simple analysis that because $X$ is a translate of $[0,1)^d$, it must be that $X$ contains one of the corners of $[0,1]^d$. And since no member of $\P_d$ contains two corners of $[0,1]^d$ (because any two corners of $[0,1]^d$ are $\ell_\infty$ distance $1$ apart, and no two points in a translate of $[0,1)^d$ are distance $1$ apart), the subset of members of $\P_d$ which intersect $[0,1]^d$ (i.e. the range of $\chi$) are in bijection with the $2^d$ corners of $[0,1]^d$.
} and consists of color sets which are rectangles\footnote{
    For each color, the set of points in $[0,1]^d$ of that color is $[0,1]^d$ intersected with some translate of $[0,1)^d$, and this intersection of a product of intervals is itself a product of intervals (i.e. a $d$-dimensional rectangle).
}. This demonstrates the existence of (finite) SLKKM colorings for which no closed $\ell_\infty$ $\epsilon$-ball intersects more than the $d+1$ colors even when taking $\epsilon$ as large as $\epsilon=\frac1{2d}$. Once more for clarity: for any $d\in\N$ and $\epsilon\in(0,\frac1{2d}]$ we have that $\clos{K}(d,\epsilon) = K^\circ(d,\epsilon) = d+1$.

\paragraph{Non-trivial upper bound:} We can generalize the result above by utilizing our newer unit cube partitions \cite[Theorem~1.9]{FOCS23_submission} rather than our initial constructions in \cite{geometry_of_rounding}. Specifically, for each $d,n\in\N$ there is a partition $\P_{d,n}$ of $\R^d$ by translates of $[0,1)^d$ which is $\left(\frac1{2n}, (n+1)^{\ceil{\frac dn}}\right)$-secluded (i.e. for every $\vec{p}\in\R^d$, the closed ball $\ballcinf{\frac1{2n}}{\vec{p}}$ intersects at most $(n+1)^{\ceil{\frac dn}}$-many members of the partition)\footnote{
    In \cite[Theorem~1.9]{FOCS23_submission}, replace ``$f(d)$'' with ``$n$.''
}. As in the prior paragraph, this immediately gives rise to an SLKKM coloring of $[0,1]^d$ such that for every $\vec{p}\in[0,1]^d$, we have that $\ballcinf{\frac1{2n}}{\vec{p}}$ contains points of at most $(n+1)^{\ceil{\frac dn}}$ colors. Thus, (along with monotonicity) we have for all $d,n\in\N$ and $\epsilon\in(0,\frac1{2n}]$ that $K^\circ(d,\epsilon)\leq\clos{K}(d,\epsilon)\leq (n+1)^{\ceil{\frac dn}}$. Taking $n=1$, we recover the trivial upper bound of $2^d$ when $\epsilon\in(0,\frac12]$ (which is tight for $\clos{K}$ at $\frac12$), and taking $n=d$ we recover the upper bound of $d+1$ when $\epsilon\in(0,\frac1{2d}]$ (which we have already said is tight on this whole interval). 

Since we can freely choose $n$, this gives some non-trivial upper bound on $\clos{K}$ and $K^\circ$ for every choice of $d$ and $\epsilon$. That is to say that for given $\epsilon\in(0,\frac12]$ and $d\in\N$, we take any $n\in\N\cap[1,\frac1{2\epsilon}]$ (in particular we can take the $n$ in this range which minimizes $(n+1)^{\ceil{\frac dn}}$) so that $\epsilon\in(0,\frac1{2n}]$ and thus by the above $K^\circ(d,\epsilon)\leq\clos{K}(d,\epsilon)\leq (n+1)^{\ceil{\frac dn}}$. Based on the fact that for any $d\in\N$ this bound is tight at the extremes $\epsilon\in(0,\frac1{2d}]$ and $\epsilon=\frac12$, we wonder if this bound is nearly tight everywhere else.

We quickly note that we really only need to consider $n\in[d]$ and not $n\in\N$ because for $n>d$, we have $(n+1)^{\ceil{\frac dn}}>d+1$ so this is a worse upper bound than $d+1$, but for $\epsilon\in(0,\frac1{2n}]$ (the domain of $\epsilon$ for which this bound applies) we already know by the ``Non-trivial tight bound for small $\epsilon$'' paragraph that that $\epsilon\in(0,\frac1{2n}\subseteq(0,\frac1{2d}]$ so $K^\circ(d,\epsilon)\leq\clos{K}(d,\epsilon)\leq d+1$. Thus, we get no new information from this bound when $n>d$. We also emphasize that it is important for $n$ to be an integer in the above bounds as is necessary in proving \cite[Theorem~1.9]{FOCS23_submission}.

\paragraph{Difference between $\boldsymbol{K^\circ}$ and $\boldsymbol{\clos{K}}$:} The final thing we will say about these functions is that $\clos{K}$ and $K^\circ$ are not in general equal,
and in fact they differ noticeably when $\epsilon=\frac12$. We have already stated that $\clos{K}(d,\frac12)=2^d$ (in the ``Trivial tight bound for large $\epsilon$'' paragraph), and we will now argue that $K^\circ(d,\frac12)\leq 2^{d-1}+1$. Consider first $d=2$ along with \Autoref{fig:hamming-colorings}; it is obviously impossible for any positioning of the {\em open} $\ell_\infty$ $\frac12$-ball to intersect all $2^d=4$ colors because it can't include both of the colors which are single points as those points are $\ell_\infty$ distance $1$ apart. This idea extends into higher dimensions by using an SLKKM coloring with $2^d$ colors which are identified with the vertices of the cube. Each of the $2^{d-1}$ colors in $\set{0,1}^d$ with even hamming weight are used only on the corresponding vertex (e.g. in \Autoref{fig:hamming-colorings} these are $\langle0,0\rangle$ (purple) and $\langle1,1\rangle$ (blue)). Every other point on the cube is given a color with odd hamming weight; this is possible because aside from vertices (which are given their own color), every other point of the cube belongs to a face which is a superset of an edge of the cube, and every edge contains both an even and odd hamming weight vertex (by the standard definition of the $d$-dimensional hypercube graph) so there is at least $1$ odd hamming weight color which can be used. This results in a KKM-style cover (but not with closed sets) which is in fact an SLKKM coloring because each point can be given a color of one of the vertices on the smallest face to which it belongs, and thus points on opposite faces are not given the same color because opposite faces are necessarily disjoint.
Such colorings have $2^{d-1}$ colors which are single points, so even the open $\ell_\infty$ $\frac12$-ball cannot hit more than one of them because all of the vertices are distance $1$ apart. Thus, the $\epsilon=\frac12$-ball can hit at most $2^{d-1}+1$ colors (possibly all of the odd hamming weight colors and one even hamming weight color), so $K^\circ(d,\frac12)\leq 2^{d-1}+1$. And similarly, for $\epsilon<\frac12$ we have $\clos{K}(d,\epsilon)\leq K^\circ(d,\frac12)\leq 2^{d-1}+1$.

In light of this, we believe it is interesting to ask for each dimension $d$ what the ranges of the functions $\clos{K}$ and $K^\circ$ are (as functions of $\epsilon$).
\begin{question}
For each $d\in\N$, what is $\range(\clos{K}(d,\cdot))$ and $\range(K^\circ(d,\cdot))$? In particular, what are the cardinalities of these ranges?
\end{question}
We know that the range is trivially a subset of the integers between $d+1$ and $2^d$ as already discussed, but now we see that it is in fact a proper subset because $K^\circ(d,\cdot)$ and $\clos{K}(d,\cdot)$ are monotonic, so the hamming coloring discussion above shows that neither range includes any values strictly between $2^{d-1}+1$ and $2^d$. We wonder if the functions $K^\circ(d,\cdot)$ and $\clos{K}(d,\cdot)$ are constant along these $d$-many open intervals:
\[
    \left(0,\frac1{2d}\right), \quad 
    \left(\frac1{2d},\frac1{2(d-1)}\right), \quad
    \left(\frac1{2(d-1)},\frac1{2(d-2)}\right), \quad
    \ldots, \quad
    \left(\frac1{8},\frac1{6}\right), \quad
    \left(\frac1{6},\frac1{4}\right), \quad
    \left(\frac1{4},\frac1{2}\right).
\]
If so, this might align nicely with the upper bounds that we obtained in the ``Non-trivial upper bound'' paragraph using \cite[Theorem~1.9]{FOCS23_submission} which gave an separate upper bound on each such interval.

We do at least know that $K^\circ(d,\cdot)$ is broken into non-trivial intervals (i.e. intervals which are not singleton sets) because it is left continuous (justified as follows). For any SLKKM coloring of $[0,1]^d$ and $\epsilon\in(0,\infty)$, there is a point $\vec{p}\in[0,1]^d$ such that $\balloinf{\epsilon}{\vec{p}}$ contains points of at least $K^\circ(d,\epsilon)$ different colors by definition. Pick $K^\circ(d,\epsilon)$-many points---each of a distinct color that is included (which is a finite number of points because $K^\circ(d,\epsilon)\leq 2^d$). Because each of these finite number of points is contained in the {\em open} ball $\balloinf{\epsilon}{\vec{p}}$, then for each $\epsilon'<\epsilon$ sufficiently large, we have that $\balloinf{\epsilon'}{\vec{p}}$ contains all of these points, and so includes at least $K^\circ(d,\epsilon)\leq 2^d$ colors. Thus, $\lim_{\epsilon'\uparrow\epsilon}K^\circ(d,\epsilon')\geq K^\circ(d,\epsilon)$ and we get the other inequality by the monotonicity of $K^\circ$ showing that $K^\circ$ is right continuous:
\[
    \lim_{\epsilon'\uparrow\epsilon}K^\circ(d,\epsilon') = K^\circ(d,\epsilon).
\]
Either by using a similar argument, or by noting that for all $\epsilon'<\epsilon$ we have $K^\circ(d,\epsilon') \leq \clos{K}(d,\epsilon') \leq K^\circ(d,\epsilon)$ we have by a squeeze theorem that also
\[
    \lim_{\epsilon'\uparrow\epsilon}\clos{K}(d,\epsilon') = K^\circ(d,\epsilon).
\]
Furthermore, on the interior of any interval $(\epsilon_a,\epsilon_b]$ for which $K^\circ(d,\cdot)$ is constant, we also know that $\clos{K}(d,\cdot)$ is constant because for $\epsilon\in(\epsilon_a,\epsilon_b)$ we have $\balloinf{\epsilon_a}{\vec{0}}\subseteq \ballcinf{\epsilon}{\vec{0}} \subseteq \balloinf{\epsilon_b}{\vec{0}}$ so $K^\circ(d,\epsilon_a)\leq \clos{K}(d,\epsilon)\leq K^\circ(d,\epsilon_b)$.

\begin{table}[h]
    \centering
    \[\def\arraystretch{1.4}
    \begin{array}{|l|l|l|l|}\hline
    \epsilon              & K^\circ(d,\epsilon)                       & \clos{K}(d,\epsilon)                      & \text{Reason} \\\hline
    \in(0,\infty)         & \leq 2^d                                  & \leq 2^d                                  & \text{Trivial}\\
    \in(\frac12,\infty)   & = 2^d                                     & = 2^d                                     & \text{Trivial} \\\hline
    =\frac12              & \leq 2^{d-1}+1                            & = 2^d                                     & \text{Hamming coloring discussion (c.f. \Autoref{fig:hamming-colorings})} \\\hline
    \in(0,\infty)         & \geq d+1                                  & \geq d+1                                  & \text{\nameref{kkm-lebesgue}} \\
    \in(0,\infty)         & \geq (1+\tfrac{\epsilon}{1+\epsilon})^d   & \geq (1+\tfrac{\epsilon}{1+\epsilon})^d   & \text{\nameref{kkm-lebesgue-variant}} \\
    \in(0,\frac12]        & \geq (1+\tfrac23\epsilon)^d               & \geq (1+\tfrac23\epsilon)^d               & \text{\nameref{kkm-lebesgue-variant}} \\\hline
    \in(0,\frac1{2d}]     & = d+1                                     & = d+1                                     & \text{\cite[Theorem~7.20]{geometry_of_rounding} \& \nameref{kkm-lebesgue}} \\
    \in(0,\frac1{2n}], n\in\N  & \leq (n+1)^{\ceil{\frac dn}}         & \leq (n+1)^{\ceil{\frac dn}}                 & \text{\cite[Theorem~1.9]{FOCS23_submission}} \\\hline
    \end{array}
    \]

    \caption{Known information about the ideal functions $K^\circ$ and $\clos{K}$.}
    \label{tab:K-bounds}
\end{table}


\begin{figure}
\centering
    \includegraphics[width=0.27\textwidth]{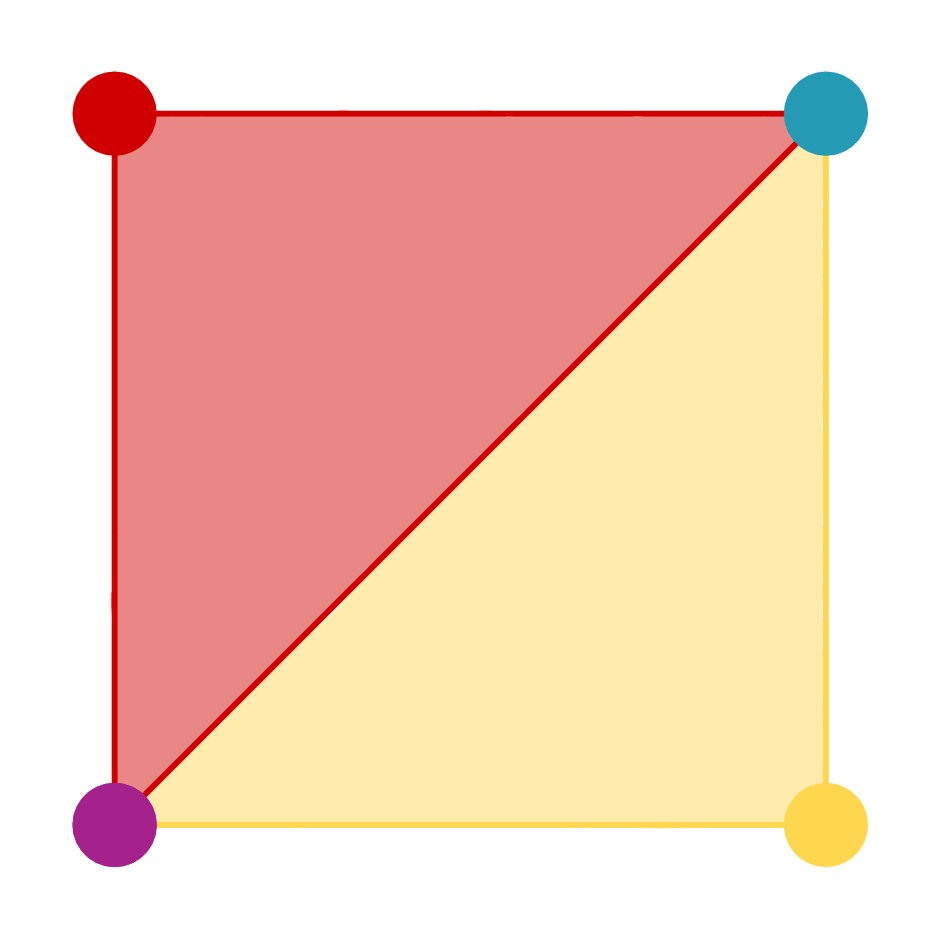}

\caption{
An example of a coloring in which each even hamming weight vertex has a color only used at that vertex. It is impossible for an $\epsilon=\frac12$ open $\ell_\infty$ ball to contain more than one of the even hamming weight colors because all vertices are distance $1$ apart.
}

\label{fig:hamming-colorings}

\end{figure}

\paragraph{Bounding the constant in the \nameref*{kkm-lebesgue-variant}:} A consequence of the ``Difference between $K^\circ$ and $\clos{K}$'' paragraph is that any improvement to the \nameref{kkm-lebesgue-variant} using a bound of the form $(1+c\epsilon)^d$ for some constant $c$ must have the property that $(1+c\frac12)^d\leq 2^{d-1}+1$ for all $d$. Solving for $c$ gives $c\leq2\left(\left(2^{d-1}+1\right)^{1/d}-1\right)$. Graphing shows that $d=3$ is the integer where this takes the smallest value and shows that $c\leq1.42$. This is not an asymptotic claim; we are only saying that if we want to replace the constant $\frac23$ in the \nameref{kkm-lebesgue-variant} with some other constant $c$, it must be that $c\leq1.42$. In particular, we cannot hope to obtain the bound of $(1+2\epsilon)^d$ (a bound which we were able to achieve in \cite{FOCS23_submission} for colorings (i.e. partitions) of $\R^d$ in which all color sets (i.e. partition members) had at most unit outer measure or at most unit diameter).

\paragraph{Conclusion:} All of this demonstrates that the lower bound of $(1+\frac{\epsilon}{1+\epsilon})^d$ or $(1+\frac23\epsilon)^d$ in our \nameref{kkm-lebesgue-variant} is not in general tight. For $\epsilon$ tending to $0$, this bound predicts $1$ even though the exact bound is $d+1$ for all $\epsilon\in(0,\frac1{2d}]$. Furthermore, our bound in the \nameref{kkm-lebesgue-variant} predicts only slightly more than $(1+\frac13)^d$ for $\epsilon$ slightly bigger than $\frac12$ when the tight bound is $2^d$. For these reasons, we suspect that the lower bound of the \nameref{kkm-lebesgue-variant} is not tight for any value of $\epsilon$, and we hope that future work is able to establish better lower bounds on $K^\circ(d,\epsilon)$ and $\clos{K}(d,\epsilon)$. In addition, we hope that future work can either improve our upper bounds on these functions or prove that they are nearly tight. 



We believe it would be nice if all of this could be done with a single technique. For example, as summarized in \Autoref{tab:K-bounds}, the current lower bounds on $K^\circ$ and $\clos{K}$ for small $\epsilon$ follow from the \nameref{true-kkm-lemma} or the \nameref{true-lebesgue-covering-theorem} (summarized in the \nameref{kkm-lebesgue}) while the current lower bounds on $K^\circ$ and $\clos{K}$ for larger $\epsilon$ given in this paper use very different techniques than those traditionally used to prove the \nameref{true-kkm-lemma} and the \nameref{true-lebesgue-covering-theorem}. So there are two very different proof techniques to prove two regimes of the current bounds. It would be  particularly satisfying if there was a single technique giving some ``nice'' lower bound expression $k^\circ(d,\epsilon)$ for $K^\circ(d,\epsilon)$ such that
\[
k^\circ(d,\epsilon) \geq \max
\begin{cases}
    \ceil{(1+\frac\epsilon{1+\epsilon})^d}\\[1em]
    d+1
\end{cases}
\]
so that it supersedes both of the two bounds we have used.
This would be nice because the \nameref{kkm-lebesgue} would then follow from this result\footnote{
    This would mean that for every $\epsilon$, there is a point $\vec{p}$ where $\balloinf{\epsilon}{\vec{p}}$ intersects at least $k^\circ(d,\epsilon)\geq d+1$ colors. 
    Consider a sequence $\langle\epsilon^{(n)}\rangle_{n=1}^\infty$ such that $\lim_{n\to\infty}\epsilon^{(n)}=0$, and for each $n$, let $\vec{p}^{(n)}\in[0,1]^d$ 
    be a point such that $\balloinf{\epsilon^{(n)}}{\vec{p^{(n)}}}$ intersects points of at least $d+1$ colors. 
    By compactness, we may assume that $\vec{p}^{(n)}$ converges (otherwise we just pick a convergent subsequence). Let $\vec{p}$ be the point of convergence. Then for any $\delta\in(0,\infty)$, it holds that $\balloinf{\delta}{\vec{p}}$ intersects at least $d+1$ colors because there is some $N\in\N$ such that for all $n\geq N$ we have $\norm{\vec{p}-\vec{p}^{(n)}}_\infty<\frac\delta2$ and $\epsilon^{(n)}<\frac\delta2$ so $\balloinf{\delta}{\vec{p}}\supseteq\balloinf{\epsilon^{(n)}}{\vec{p}^{(n)}}$ which intersects at least $d+1$ colors.
    It is left to show that $\vec{p}$ is in fact at the closure of some set of at least $d+1$ colors, and finiteness is critical for this. Consider a sequence $\langle\delta^{(n)}\rangle_{n=1}^\infty$ converging to $0$ and the sequence $\langle C^{(n)}\rangle_{n=1}^\infty$ where $C_n$ is the set of colors intersected by $\balloinf{\delta^{(n)}}{\vec{p}}$. Because only finitely many colors are used by hypothesis of the \nameref{kkm-lebesgue}, there are at most a finite number of distinct terms in $\langle C^{(n)}\rangle_{n=1}^\infty$ as each term is a subset of the colors. Thus, there is some set of colors $C$ which appears infinitely many times in the sequence, and since each $C^{(n)}$ has cardinality at least $d+1$, so does $C$. That is, we have found a sequence of arbitrarily small open balls at $\vec{p}$ which each intersect all $d+1$ colors in $C$, and so $\vec{p}$ is in the closure of all colors in $C$.
}; equivalently (as shown in \Autoref{subsec:kkm-lebesgue-coloring-equivalence}) the \nameref{true-lebesgue-covering-theorem} and the \nameref{true-kkm-lemma} would follow from this result and so both the \nameref{true-lebesgue-covering-theorem} and the \nameref{true-kkm-lemma} could be viewed as special cases of a more general result.

\section{Acknowledgements}
Pavan's work is partially supported by NSF award 2130536.  Vander Woude and Vinodchandran's work is partially supported by NSF award 2130608. Radcliffe's work is partially supported by Simons grant 429383.
Part of the work was done when Pavan and Vinodchandran were visiting the Simons Institute for the Theory of Computing. Part of the work was done while Dixon was a postdoc at Ben-Gurion University of the Negev. 

\bibliography{references}
\bibliographystyle{plainurl}

\newpage
\appendix

\section{Skipped Proofs}
\label{sec:appendix}

\subsection{Equivalence of the KKM Lemma, Lebesgue Covering Theorem, and Color Variant}\label{subsec:kkm-lebesgue-coloring-equivalence}

\begin{lemma}[KKM/Lebesgue $\Longrightarrow$ KKM]
The \namefullref{kkm-lebesgue} implies the \namefullref{true-kkm-lemma}.
\end{lemma}
\begin{proof}
Let $\C=\set{C_{\vec{v}}}_{\vec{v}\in\set{0,1}^d}$ be a KKM cover of $[0,1]^d$. For each $\vec{x}\in[0,1]^d$, let $F_{\vec{x}}$ denote the smallest face of the cube containing $\vec{x}$ (i.e. $F_{\vec{x}}$ is the intersection of all faces containing $\vec{x}$). By the defining property of a KKM cover, we have $F_{\vec{x}}\subseteq\bigcup_{\vec{v}\in F_{\vec{x}}\cap\set{0,1}^d}C_{\vec{v}}$, so in particular there exists some $\vec{v}\in F_{\vec{x}}\cap\set{0,1}^d$ with $\vec{x}\in C_{\vec{v}}$. Define the function $\chi$ as follows where $\min_{\mathrm{lex}}$ denotes the minimum element in a subset of $\set{0,1}^d$ under the lexicographic ordering:
\begin{align*}
    \chi &:[0,1]^d \to \set{0,1}^d\\
    \chi(\vec{x}) &= \min_{\mathrm{lex}}\set{\vec{v}\in\set{0,1}^d\cap F_{\vec{x}}: \vec{x}\in C_{\vec{v}}}
\end{align*}
We have already demonstrated that the the set in the definition is not empty, so $\chi$ is well-defined.

We claim that $\chi$ is a finite SLKKM coloring of $[0,1]^d$. The finiteness is trivial because the codomain of $\chi$ is finite, so we need only show it is an SLKKM coloring. Suppose $F^{(0)}$ and $F^{(1)}$ are opposite faces of the cube (i.e. there is some coordinate $j\in[d]$ such that $\pi_j(F^{(0)})=\set{0}$ and $\pi_j(F^{(1)})=\set{1}$) and let $\vec{x}^{(0)}\in F^{(0)}$ and $\vec{x}^{(1)}\in F^{(1)}$. Because $\pi_j(F^{(0)})\cap \pi_j(F^{(1)})=\emptyset$, it follows that $F^{(0)}\cap F^{(1)}=\emptyset$, so $F^{(0)}$ and $F^{(1)}$ are disjoint sets.

Because $\vec{x}^{(0)}\in F^{(0)}$ and $F_{\vec{x}^{(0)}}$ is by definition the intersection of all faces containing $\vec{x}$, we have $F_{\vec{x}^{(0)}}\subseteq F^{(0)}$ (and similarly replacing ``$0$'' with ``$1$'') so that also $F^{(0)}$ and $F^{(1)}$ are disjoint. By definition of $\chi$ we have $\chi(\vec{x}^{(0)})\in F^{(0)}$ and $\chi(\vec{x}^{(1)})\in F^{(1)}$ showing that $\chi(\vec{x}^{(0)})\not=\chi(\vec{x}^{(1)})$, so $\chi$ is an SLKKM coloring.

By the \nameref{kkm-lebesgue}, there exists $\vec{p}\in[0,1]^d$ such that $\abs{\set{\vec{v}\in\set{0,1}^d:\vec{p}\in\clos{\chi^{-1}(\vec{v})}}}\geq d+1$. Fix such a $\vec{p}$ for the remainder of the proof. For each $\vec{v}\in\set{0,1}^d$, observe that $\chi^{-1}(\vec{v})\subseteq C_{\vec{v}}$ because for any $\vec{x}\in\chi^{-1}(\vec{v})$ we have $\chi(\vec{x})=\vec{v}$, so by definition of $\chi$ we have $\vec{x}\in C_{\vec{v}}$. Because closures maintain subset containment and because $C_{\vec{v}}$ is a closed set by hypothesis of the \nameref{true-kkm-lemma}, we have $\clos{\chi^{-1}(\vec{v})}\subseteq\clos{C_{\vec{v}}}=C_{\vec{v}}$. It then follows immediately that
\[
\set{\vec{v}\in\set{0,1}^d: \vec{p}\in\chi^{-1}(\vec{v})} \subseteq \set{\vec{v}\in\set{0,1}^d: \vec{p}\in C_{\vec{v}}}
\]
and since the former has cardinality at least $d+1$, so does the latter which proves the \nameref{true-kkm-lemma}.
\end{proof}


\begin{lemma}[KKM/Lebesgue $\Longrightarrow$ Lebesgue]
The \namefullref{kkm-lebesgue} implies the \namefullref{true-lebesgue-covering-theorem}.
\end{lemma}
\begin{proof}
Let $N\in\N$ and $\C=\set{C_n}_{n\in[N]}$ be a Lebesgue cover of $[0,1]^d$. Because this is a cover, every point of $[0,1]^d$ belongs to some set, so define $\chi$ as follows:
\begin{align*}
    \chi &: [0,1]^d \to [N]\\
    \chi(\vec{x}) &= \min\set{n\in[N]: \vec{x}\in C_{n}}.
\end{align*}
This is trivially a finite SLKKM coloring of $[0,1]^d$ because the codomain of $\chi$ is finite and for $\vec{x}^{(0)}$ and $\vec{x}^{(1)}$ on opposite faces, there is no $n\in[N]$ for which both $\vec{x}^{(0)}\in C_n$ and $\vec{x}^{(1)}\in C_n$ and thus $\set{n\in[N]: \vec{x}^{(0)}\in C_{n}}$ and $\set{n\in[N]: \vec{x}^{(1)}\in C_{n}}$ are disjoint, so $\chi(\vec{x}^{(0)})\not=\chi(\vec{x}^{(1)})$.

By the \nameref{kkm-lebesgue}, there exists $\vec{p}\in[0,1]^d$ such that $\abs{\set{n\in[N]:\vec{p}\in\clos{\chi^{-1}(n)}}}\geq d+1$. Fix such a $\vec{p}$ for the remainder of the proof. For each $n\in[N]$, observe that $\chi^{-1}(n)\subseteq C_n$ because for any $\vec{x}\in\chi^{-1}(n)$ we have $\chi(\vec{x})=n$, so by definition of $\chi$ we have $\vec{x}\in C_n$. Because closures maintain subset containment and because $C_{n}$ is a closed set by hypothesis of the \nameref{true-lebesgue-covering-theorem}, we have $\clos{\chi^{-1}(n)}\subseteq\clos{C_n}=C_{n}$. It then follows immediately that
\[
\set{n\in[N]: \vec{p}\in\chi^{-1}(n)} \subseteq \set{n\in[N]: \vec{p}\in C_n}
\]
and since the former has cardinality at least $d+1$, so does the latter which proves the \nameref{true-lebesgue-covering-theorem}.
\end{proof}


\begin{lemma}[Lebesgue $\Longrightarrow$ KKM/Lebesgue]
The \namefullref{true-lebesgue-covering-theorem} implies the \namefullref{kkm-lebesgue}.
\end{lemma}
This proof is probably the trickiest of the four. In order to use the hypothesis of the \namefullref{true-lebesgue-covering-theorem}, we can't just close the sets in an SLKKM coloring because the closures might intersect opposite faces. Thus, we first have to extend the coloring, and we do so as we do in the proof of the main result of the paper (the \namefullref{kkm-lebesgue-variant}).
\begin{proof}
For this proof, we assume that the cube is $[-\frac12,\frac12]^d$ instead of $[0,1]^d$. Let $C$ be a finite set and $\chi:[-\frac12,\frac12]^d\to C$ a finite SLKKM coloring of $[-\frac12,\frac12]^d$. Let $\epsilon\in(0,\infty)$ be any fixed value throughout the entirety of the proof. Let $f$ and $\gamma=\chi\circ f$ as in the proof of the \namefullref{kkm-lebesgue-variant} so that $\gamma:[-\frac12-\epsilon,\frac12+\epsilon]^d\to C$ is an extension of the coloring $\chi$ to the larger cube $\gamma:[-\frac12-\epsilon,\frac12+\epsilon]^d$ (see also \Autoref{subfig:1_non_spanning_coloring} and \Autoref{subfig:2_extended_coloring}) with the property that for each color $c\in C$, there exists an orientation $\vec{v}^{(c)}\in\set{-1,1}^d$ so that the set $Y_c\defeq \gamma^{-1}(c)$ of points of color $c$ (according to $\gamma$) is a subset of $\prod_{i=1}^d v^{(c)}_i\cdot(-\frac12,\frac12+\epsilon]$ (see \Autoref{kkm-lebesgue-variant-subclaim-1} in the proof of the \namefullref{kkm-lebesgue-variant}). Because closures maintain subset containment we have
\[
\clos{\gamma^{-1}(c)} = \clos{Y_c} \subseteq \clos{\prod_{i=1}^d v^{(c)}_i\cdot(-\tfrac12,\tfrac12+\epsilon]} = \prod_{i=1}^d v^{(c)}_i\cdot[-\tfrac12,\tfrac12+\epsilon].
\]

This demonstrates that for each color $c\in C$, the set $\gamma^{-1}(c)$ of points given color $c$ does not include points on opposite faces of the cube $[-\frac12-\epsilon,\frac12+\epsilon]^d$ (because if it did, then there would be some coordinate $j\in[d]$ such that $\pi_j(\gamma^{-1}(c))\supseteq\set{-\frac12-\epsilon,\frac12+\epsilon}$, but the containment above shows this is not the case). Thus, $\C=\set{\gamma^{-1}(c)}_{c\in\C}$ is a Lebesgue cover of the cube $[-\frac12-\epsilon,\frac12+\epsilon]^d$ (we could rescale it to $[0,1]^d$ and re-index the cover with $[N]$ for $N=\abs{C}$ to be really formal, but we won't).

By the \nameref{true-lebesgue-covering-theorem}, there exists $\vec{q}\in[-\frac12-\epsilon,\frac12+\epsilon]^d$ such that $\abs{\set{c\in\C:\vec{q}\in\clos{\gamma^{-1}(c)}}}\geq d+1$. Fix such a $\vec{q}$ for the remainder of the proof and let $\vec{p}=f(\vec{q})$ (recalling that $f:[-\frac12-\epsilon,\frac12+\epsilon]^d\to[-\frac12,\frac12=^d$ is as in the proof of the \namefullref{kkm-lebesgue-variant}).

\begin{subclaim}
    For each $c\in C$, if $\vec{q}\in\clos{\gamma^{-1}(c)}$, then $\vec{p}\in\clos{\chi^{-1}(c)}$.
\end{subclaim}
\begin{subclaimproof}
Let $c\in C$ be arbitrary. If $\vec{q}\in\clos{\gamma^{-1}(c)}$, then there is a sequence $\langle \vec{q}^{(j)} \rangle_{j=1}^\infty$ of points in $\gamma^{-1}(c)$ converging to $\vec{q}$. Because for each $j\in\N$ we have $\vec{q}^{(j)}\in\gamma^{-1}(c)$, we have $\gamma(\vec{q}^{(j)})=c$. Then, by the definition of $\gamma$, we have
\[
    c = \gamma(\vec{q}^{(j)}) = \chi(f(\vec{q}^{(j)}))
\]
showing that $f(\vec{q}^{(j)})\in\chi^{-1}(c)$. Since $f$ is a continuous function\footnote{
    It is argued implicitly in \Autoref{kkm-lebesgue-variant-subclaim-4} in the proof of the \namefullref{kkm-lebesgue-variant} that $f$ is Lipschitz with Lipschitz constant $1$. Alternatively, this could be analyzed directly.
}, then $\langle f(\vec{q}^{(j)}) \rangle_{j=1}^\infty$ converges to $f(\vec{q})=\vec{p}$ demonstrating that $\vec{p}\in\clos{\chi^{-1}(c)}$.
\end{subclaimproof}

It then follows immediately that
\[
    \set{c\in C: \vec{p}\in\clos{\gamma^{-1}(c)}} \subseteq \set{c\in C: \vec{p}\in \clos{\chi^{-1}(c)}}
\]
and since the former has cardinality at least $d+1$ (established prior to the claim), so does the latter which proves the \nameref{kkm-lebesgue}.
\end{proof}


\begin{lemma}[KKM $\Longrightarrow$ KKM/Lebesgue]\label{kkm-implies-kkm-lebesgue}
The \namefullref{true-kkm-lemma} implies the \namefullref{kkm-lebesgue}.
\end{lemma}
In the proof we essentially condense an SLKKM coloring to have codomain of cardinality $2^d$---one color associated to each vertex---and then close each color set to apply the \namefullref{true-kkm-lemma}.
\begin{proof}
Let $C$ be a finite set and $\chi:[0,1]^d\to C$ a finite SLKKM coloring. Because $\chi$ does not map points on opposite faces to the same color, this means for each color $c\in C$ and coordinate $i\in[d]$ that the set of points given color $c$ (i.e. $\chi^{-1}(c)$) does not contain a point with $i$th coordinate $0$ and also a point with $i$th coordinate $1$ (i.e. $\pi_{i}(\chi^{-1}(c))\not\supseteq\set{0,1}$).

For each $i\in[d]$, define $f_{i}:C\to\set{0,1}$ by
\[
    f_{i}(c) =
    \begin{cases}
      0 & 0\in\pi_{i}(\chi^{-1}(c))\\
      1 & 1\in\pi_{i}(\chi^{-1}(c))\\
      0 & \text{otherwise}
    \end{cases}.
\]
The function $f_{i}$ is well-defined because the first two cases are mutually exclusive. Then define $f:C\to\set{0,1}^{d}$ by $f(c)=\langle f_{i}(c)\rangle_{i=1}^{d}$, and define the (coloring) function $\zeta:[0,1]^{d}\to\set{0,1}^{d}$ as the composition $f\circ\chi$.

For each $\vec{v}\in\set{0,1}^d$, let $C_{\vec{v}}=\clos{\zeta^{-1}(\vec{v})}$, and let $\C=\set{C_{\vec{v}}}_{\vec{v}\in\set{0,1}^d}$. We claim that $\C$ is a KKM cover of $[0,1]^d$ which we prove by the following claim.

\begin{subclaim}
    For each face $F$ of $[0,1]^d$, we have $F\subseteq\bigcup_{\vec{v}\in F\cap\set{0,1}^d} C_{\vec{v}}$.
\end{subclaim}
\begin{subclaimproof}
    Let $F$ be an arbitrary face of $[0,1]^d$; this means that $F=\prod_{i=1}^d F_i$ where each $F_i$ is either $\set{0}$, $\set{1}$, or $[0,1]$. Let $\vec{x}\in F$ be arbitrary noting that this implies for each coordinate $i\in[d]$ that $x_i\in F_i$. Let $c=\chi(\vec{x})$ (so $\vec{x}\in\chi^{-1}(c))$.
    
    We first show for each coordinate $i\in[d]$ that $f_i(c)\in F_i$ and do so in three cases.
    \begin{enumerate}
        \item If $x_i=0$, then $0=x_i\in\pi_i(\chi^{-1}(c))$, so by definition of $f_i$ we have $f_i(c)=0$ showing that $f_i(c)=0=x_i\in F_i$.
        \item If $x_i=1$, then $1=x_i\in\pi_i(\chi^{-1}(c))$, so by definition of $f_i$ we have $f_i(c)=1$ showing that $f_i(c)=1=x_i\in F_i$.
        \item Otherwise $x_i\in(0,1)$, so because $x_i\in F_i$ we cannot have $F_i=\set{0}$ or $F_i=\set{1}$ and so it must be that $F_i=[0,1]$. Thus, $f_i(c)\in\set{0,1}\subseteq F_i$.
    \end{enumerate}
    
    Now let $\vec{v}^{(0)}=\zeta(\vec{x})$ (so $\vec{x}\in\zeta^{-1}(\vec{v}^{(0)})$) and observe the following:
    \[
        \vec{v}^{(0)} = \zeta(\vec{x}) = f(\chi(\vec{x})) = f(c) = \langle f_i(c) \rangle_{i=1}^d \in \prod_{i=1}^d F_i = F.
    \]
    Thus $\vec{v}^{(0)}\in F$, and also vacuously $\vec{v}^{(0)}=\zeta(\vec{x})\in\set{0,1}^d$. This shows that
    \[
    \vec{x} \in \zeta^{-1}(\vec{v}^{(0)}) \subseteq \clos{\zeta^{-1}(\vec{v}^{(0)})} = C_{\vec{v}^{(0)}} \subseteq \bigcup_{\vec{v}\in F\cap\set{0,1}^d} C_{\vec{v}}.
    \]
    Since $\vec{x}\in F$ was arbitrary, we have $F \subseteq \bigcup_{\vec{v}\in F\cap\set{0,1}^d} C_{\vec{v}}$ as claimed.
\end{subclaimproof}

Because $\C$ is a KKM cover, by the \nameref{true-kkm-lemma}, there exists $\vec{p}\in[0,1]^d$ such that the set $V'\defeq\set{\vec{v}\in\set{0,1}^d:\vec{p}\in C_{\vec{v}}}$ has cardinality at least $d+1$. Fix such a $\vec{p}$ for the remainder of the proof. 

Note that for each $\vec{v}\in V$, we have
  \begin{equation}\label{eq:historic-sperner-kkm-lebesgue}
    \zeta^{-1}(\vec{v}) = (f\circ \chi)^{-1}(\vec{v}) = \chi^{-1}(f^{-1}(\vec{v})) = \bigcup_{c\in f^{-1}(\vec{v})}\chi^{-1}(c).
  \end{equation}

Now, for each $\vec{v}\in V'$, because $\vec{p}$ is in the closure of $\zeta^{-1}(\vec{v})$, any open set containing $\vec{p}$ intersects $\zeta^{-1}(\vec{v})=\bigcup_{c\in f^{-1}(\vec{v})}\chi^{-1}(c)$ and thus intersects $\chi^{-1}(c)$ for some $c\in f^{-1}(\vec{v})$. Let $g(\vec{v})$ denote one such color\footnote{
    We don't need the full axiom of choice here because $C$ has finite cardinality.
}.

Because $f^{-1}(\vec{v})$ and $f^{-1}(\vec{v}\,')$ are trivially disjoint for $\vec{v}\not=\vec{v}\,'$, this means $g(\vec{v})$ and $g(\vec{v}\,')$ are distinct colors so $g:V'\to C$ is an injection which means there are at least $d+1$ colors in $C$ that are intersected by any open set containing $\vec{p}$.

Because $\abs{C}$ is finite, then for each $\vec{v}\in V$, $f^{-1}(\vec{v})\subseteq C$ is finite, then we can use the fact the the closure of a finite union is equal to the finite union of the closures to extend this to
  \begin{align*}
    \vec{p} &\in \bigcap_{\vec{v}\in V'} \clos{\zeta^{-1}(\vec{v})}\\
            &= \bigcap_{\vec{v}\in V'} \clos{\bigcup_{c\in f^{-1}(\vec{v})}\chi^{-1}(c)}\\
            &= \bigcap_{\vec{v}\in V'} \bigcup_{c\in f^{-1}(\vec{v})}\clos{\chi^{-1}(c)}\tag{$f^{-1}(\vec{v})$ is finite}
  \end{align*}
and thus, for each $\vec{v}\in V'$, $\vec{p}$ belongs to the closure of $\chi^{-1}(c)$ for some $c\in f^{-1}(\vec{v})$. By the same argument there are at least $d+1$ such colors in $C$. That is, $\abs{\set{c\in C: \vec{p}\in\clos{\chi^{-1}(c)}}}\geq d+1$ which proves the \nameref{kkm-lebesgue}.
\end{proof}

\subsection{Background for the \nameref*{kkm-lebesgue-variant}}
\label{appendix-sub:brunn-minkowski}

\restatableSmallerCeiling
\begin{proof}
Let $n=\ceil{\alpha}$. This implies that $\alpha>n-1$ (otherwise $\alpha\leq n-1$ so $\ceil{\alpha}\leq n-1$). Thus $(n-1,\alpha)$ is non-empty and we can take any $\gamma\in(n-1,\alpha)$. Then $n-1<\gamma\leq\ceil{\gamma}\leq\ceil{\alpha}=n$ showing that $\ceil{\gamma}=n$ as well.
\end{proof}


\restatableBrunnMinkowskiBoundLemma
\begin{proof}
We will show that $(x^{1/d}+\alpha)^d - x(1+\alpha)^d\geq 0$ for these parameters. Let For $d$, $\alpha$ as above, let $f:[0,1]\to\R$ be defined by $f(x)=(x^{1/d}+\alpha)^d - x(1+\alpha)^d$. Observe that $f(0)=\alpha^d\geq0$ and $f(1)=(1+\alpha)^d-(1+\alpha)^d=0$. We will now prove that $f$ is convex on the domain\footnote{
    Actually we could have defined the domain of $f$ to be $[0,\infty)$ and we show that $f$ is convex on that domain. However, we only have need of the interval $[0,1]$ because that is where we know that $f$ is non-negative.
} $[0,1]$. This will be sufficient to prove the claim because $f$ is also non-negative at $0$ and at $1$.

We show that $f$ is convex on $[0,1]$ by considering its second derivative on $(0,1]$. We start with the first derivative with respect to $x$.
\begin{align*}
    \frac{d}{dx}f(x) & = \frac{d}{dx} \left[(x^{1/d}+\alpha)^d - x(1+\alpha)^d\right] \\
    &= d(x^{1/d}+\alpha)^{d-1}\cdot\frac{1}{d}x^{(1/d)-1} - (1+\alpha)^d \\
    &= (x^{1/d}+\alpha)^{d-1}x^{(1/d)-1} - (1+\alpha)^d \\
\end{align*}
where we use the standard convention that $0^0=1$. Then
\begin{align*}
    \frac{d^2}{dx^2}f(x) & = \frac{d}{dx} \left[(x^{1/d}+\alpha)^{d-1}x^{1/d-1} - (1+\alpha)^d\right] \\
    &= (x^{1/d}+\alpha)^{d-1}\cdot\left(\frac1d-1\right)x^{1/d-2}+x^{1/d-1}\cdot(d-1)(x^{1/d}+\alpha)^{d-2}\frac1d x^{1/d-1} \\
    &= (x^{1/d}+\alpha)(x^{1/d}+\alpha)^{d-2}\left(-\frac{d-1}d\right)\left(x^{1/d-2}\right) + (x)\left(x^{1/d-2}\right)\left(\frac{d-1}{d}\right)\left(x^{1/d}+\alpha\right)^{d-2}\left(x^{1/d-1}\right) \\
    &= \frac{(x^{1/d}+\alpha)^{d-2}(d-1)\left(x^{1/d-2}\right)}{d}\left[ -(x^{1/d}+\alpha) + (x)\left(x^{1/d-1}\right) \right] \\
    &= \frac{(x^{1/d}+\alpha)^{d-2}(d-1)\left(x^{1/d-2}\right)}{d}\left[ -(x^{1/d}+\alpha) + \left(x^{1/d}\right) \right] \\
    &= \frac{-\alpha(x^{1/d}+\alpha)^{d-2}(d-1)\left(x^{1/d-2}\right)}{d} \\
    &= \frac{-\alpha(x^{1/d}+\alpha)^{d}(d-1)\left(x^{1/d}\right)}{d(x^{1/d}+\alpha)^{2}x^2}.
\end{align*}
Note that because $\alpha\geq0$ and $d\geq1$, we have for $x\in(0,\infty]$ that the denominator is positive and the numerator is non-positive showing that $\frac{d^2}{dx^2}f(x)\leq0$ so $f$ is convex on $(0,\infty)$ and by continuity on $[0,\infty)$. Because $f(x)$ is convex on $[0,1]$ and non-negative at $x=0$ and $x=1$, it is non-negative on all of $[0,1]$.
\end{proof}


In order to prove \Autoref{outer-measure-brunn-minkowsi-bound-l-infty} we need some simple background results.

The following fact says that the Minkowski sum of a set $X$ and an open ball at the origin can be viewed as a union of open balls positioned at each point of $X$. This is useful not only conceptually but because it guarantees that this Minkowski sum is open and thus measurable.

\begin{restatable}{fact}{restatableMinkowskiBubbleUnion}\label{minkowski-bubble-union}
For any normed vector space, given a set $X$ and $\epsilon\in(0,\infty)$, then
\[
X+\ballo{\epsilon}{\vec{0}} = \bigcup_{\vec{x}\in X}\ballo{\epsilon}{\vec{x}}.
\]
The same can be said replacing open balls with closed balls.
\end{restatable}
\begin{proof}
We show this only for the open balls. Switching all open balls in the proof with closed ones gives the proof for closed balls.

($\subseteq$) A generic element of $X+\ballo{\varepsilon}{\vec{0}}$ is $\vec{x}+\vec{b}$ for some $\vec{x}\in X$ and $\vec{b}\in \ballo{\varepsilon}{\vec{0}}$ which means $\vec{x}+\vec{b}\in\vec{x}+\ballo{\varepsilon}{\vec{0}}=\ballo{\varepsilon}{\vec{x}}\subseteq \bigcup_{\vec{x}\in X}\ballo{\varepsilon}{\vec{x}}$.

($\supseteq$) Given $\vec{y}\in \bigcup_{\vec{x}\in X}\ballo{\varepsilon}{\vec{x}}$ there is some particular $\vec{x}\in X$ such that $y\in\ballo{\varepsilon}{\vec{x}}=\vec{x}+\ballo{\varepsilon}{\vec{0}}$ which means $\vec{y}=\vec{x}+\vec{b}$ for some $\vec{b}\in \ballo{\varepsilon}{\vec{0}}$, and since $\vec{x}\in X$, we have $\vec{y}\in X+\ballo{\varepsilon}{\vec{0}}$.
\end{proof}


The next fact (\Autoref{sum-of-balls}) says that we can decompose a ball into a Minkowski sum of two smaller balls. We will use this along with associativity of the Minkowski sum to deal swiftly with issues of non-measurable sets.

\begin{restatable}{fact}{restatableSumOfBalls}\label{sum-of-balls}
For any normed vector space, and any $\alpha,\beta\in(0,\infty)$, it holds that $\ballo{\alpha}{\vec{0}}+\ballo{\beta}{\vec{0}}=\ballo{\alpha+\beta}{\vec{0}}$.
\end{restatable}
\begin{proof}
($\subseteq$) A generic element of $\ballo{\alpha}{\vec{0}}+\ballo{\beta}{\vec{0}}$ is $\vec{a}+\vec{b}$ for $\vec{a}\in\ballo{\alpha}{\vec{0}}$ and $\vec{b}\in\ballo{\beta}{\vec{0}}$. Then $\norm{\vec{a}}<\alpha$ and $\norm{\vec{b}}<\beta$, so $\norm{\vec{a}+\vec{b}}<\alpha+\beta$ showing $\vec{a}+\vec{b}\in\ballo{\alpha+\beta}{\vec{0}}$.

($\supseteq$) Let $\vec{x}\in\ballo{\alpha+\beta}{\vec{0}}$ which implies $\norm{\vec{x}}<\alpha+\beta$. If $\norm{\vec{x}}<\alpha$, then $\vec{x}\in\ballo{\alpha}{\vec{0}}$ and $\vec{0}\in\ballo{\beta}{\vec{0}}$, so $\vec{x}=\vec{x}+\vec{0}\in\ballo{\alpha}{\vec{0}}+\ballo{\beta}{\vec{0}}$. Similarly, if $\norm{\vec{x}}<\beta$, then $\vec{x}=\vec{0}+\vec{x}\in\ballo{\alpha}{\vec{0}}+\ballo{\beta}{\vec{0}}$. In either case we would be done, so we may now assume that $\norm{\vec{x}}\geq\alpha,\beta$. Let $\varepsilon=\alpha+\beta-\norm{\vec{x}}\in(0,\infty)$. Since $\norm{\vec{x}}\geq\alpha$, we have $\varepsilon\leq\beta$, and because $\norm{\vec{x}}\geq\beta$, we have $\varepsilon\leq\alpha$. This shows $\frac\varepsilon2<\alpha,\beta$. Let $\vec{a}=(\alpha-\frac\varepsilon2)\frac{\vec{x}}{\norm{\vec{x}}}$ and $\vec{b}=(\beta-\frac\varepsilon2)\frac{\vec{x}}{\norm{\vec{x}}}$ noting that $\norm{\vec{a}}=\alpha-\frac\varepsilon2\in(0,\alpha)$ and $\norm{\vec{b}}=\beta-\frac\varepsilon2\in(0,\beta)$. Also, note that $\vec{a}+\vec{b}=(\alpha-\frac\varepsilon2+\beta-\frac\varepsilon2)\frac{\vec{x}}{\norm{\vec{x}}}=\norm{\vec{x}}\frac{\vec{x}}{\norm{\vec{x}}}=\vec{x}$ which shows $\vec{x}\in\ballo{\alpha}{\vec{0}}+\ballo{\beta}{\vec{0}}$.
\end{proof}


The final fact (\Autoref{antisymmetry-containment-pre-claim-minkowski-sum}), while also very simple, is the key change of perspective that allowed us to prove the main result. It says that if we are checking if a member $X$ in our partition intersects an $\epsilon$-ball located at $\vec{p}$ (in order to see how many such members there are), we can instead enlarge $X$ by taking its Minkowski sum with the origin-centered $\epsilon$-ball, and check if this enlarged member contains the point $\vec{p}$.

\begin{restatable}{fact}{restatableAntisymmetryContainmentPreClaimMinkowskiSum}\label{antisymmetry-containment-pre-claim-minkowski-sum}
For any normed vector space, for any set $X$ and vector $\vec{p}$, the following are equivalent:
\begin{enumerate}
    \item $\ballc{\epsilon}{\vec{p}}\cap X\not=\emptyset$
    \item $\vec{p}\in X+\ballc{\epsilon}{\vec{0}}$
\end{enumerate}
The same can be said replacing both closed balls with open balls.
\end{restatable}
\begin{proof}
We show this only for the closed balls. Switching all closed balls in the proof with open ones gives the proof for open balls.

($\Longrightarrow$) If $\ballc{\varepsilon}{\vec{p}}\cap X\not=\emptyset$, then there exists $\vec{y}\in\ballc{\varepsilon}{\vec{p}}\cap X$, and since $\vec{y}\in\ballc{\varepsilon}{\vec{p}}$ we have $\norm{\vec{y}-\vec{p}}\leq\varepsilon$ so $\vec{p}\in\ballc{\varepsilon}{\vec{y}}=\vec{y}+\ballc{\varepsilon}{\vec{0}}$, and since $\vec{y}\in X$, $\vec{y}+\ballc{\varepsilon}{\vec{0}}\subseteq X+\ballc{\varepsilon}{\vec{0}}$ showing that $\vec{p}\in X+\ballc{\varepsilon}{\vec{0}}$.

($\Longleftarrow$) If $\vec{p}\in X+\ballc{\varepsilon}{\vec{0}}$ then there exists $\vec{x}\in X$ and $\vec{b}\in\ballc{\varepsilon}{\vec{0}}$ such that $\vec{p}=\vec{x}+\vec{b}$. Thus $\vec{p}-\vec{x}=\vec{b}$, so $\norm{\vec{p}-\vec{x}}=\norm*{\vec{b}}\leq\varepsilon$, so $\vec{x}\in\ballc{\varepsilon}{\vec{p}}$. Since $\vec{x}$ belongs to both $X$ and $\ballc{\varepsilon}{\vec{p}}$, their intersection is non-empty.
\end{proof}


With this background, we can state the \nameref{generalized-brunn-minkowski-inequality} and prove \Autoref{outer-measure-brunn-minkowsi-bound-l-infty}. The \nameref*{generalized-brunn-minkowski-inequality} gives a bound on the volume of Minkowski sums. It is common in measure theory for $d$th roots of $d$-dimensional volumes to show up and they serve as a type of characteristic length scale. With this interpretation, the \nameref*{generalized-brunn-minkowski-inequality} says that the characteristic length scale of a Minkowski sum is no less than the sum of characteristic length scales of the parts.

\begin{restatable}[Generalized Brunn-Minkowski Inequality]{theorem-known}{restatable-generalized-brunn-minkowski-inequality}\label{generalized-brunn-minkowski-inequality}
Let $d\in\N$ and $A,B\subseteq\R^d$ be non-empty and Lebesgue measurable such that $A+B$ is also Lebesgue measurable. Then
\[
m(A+B)^\frac1d \geq m(A)^\frac1d + m(B)^\frac1d.
\]
\end{restatable}

This version of the statement can be obtained from \cite[Equation~11]{gardner_brunn-minkowski_2002}; in that survey, Gardner states this theorem with a requirement that the sets be bounded, but in the following paragraph notes that this is not necessary and the requirement is only stated for convenience of the presentation in that survey. In the theorem, the requirement that $A+B$ is Lebesgue measurable is not a triviality; Gardner discusses that there exist known Lebesgue measurable sets $A$ and $B$ such that the Minkowski sum $A+B$ is not Lebesgue measurable as shown in \cite{sierpinski_sur_1920}. The next result gives us a way to circumvent this issue in our application even if the members of our partition are not measurable by taking $B$ to be an open set so that the sum $A+B$ is open (and thus measurable), and using the outer measure of $A$ so that we don't need the assumption that $A$ is measurable.


Now we can prove the result which immediate implies \Autoref{outer-measure-brunn-minkowsi-bound-l-infty} when we instantiate it with the $\ell_\infty$ norm specifically.

\begin{lemma}\label{outer-measure-brunn-minkowsi-bound}
Let $d\in\N$ and let $\R^d$ be equipped with any norm $\norm{\cdot}$. Let $Y\subseteq\R^d$, and $\epsilon\in(0,\infty)$. Then $Y+\ballostd{\epsilon}{\vec{0}}$ is open (and thus Borel measurable), and $m(Y+\ballostd{\epsilon}{\vec{0}})\geq \left(m_{out}(Y)^\frac1d+\epsilon\cdot (\unitballmeasurestd)^\frac1d\right)^d$.
\end{lemma}
\begin{proof}
By \Autoref{minkowski-bubble-union}, $Y+\ballostd{\epsilon}{\vec{0}}=\bigcup_{\vec{y}\in Y}\ballostd{\epsilon}{\vec{y}}$  which is a union of open sets, so is itself open and thus Borel measurable. Now, for any $\epsilon'\in(0,\epsilon)$, observe that by \Autoref{sum-of-balls}, $\ballostd{\epsilon}{\vec{0}}=\ballostd{\epsilon-\epsilon'}{\vec{0}}+\ballostd{\epsilon'}{\vec{0}}$ and thus, this sum is measurable because it is an open ball. Using this equality and the associativity of the Minkowski sum, we have 
\[
Y+\ballostd{\epsilon}{\vec{0}} = Y + \left[\ballostd{\epsilon-\epsilon'}{\vec{0}}+\ballostd{\epsilon'}{\vec{0}}\right] = \left[Y + \ballostd{\epsilon-\epsilon'}{\vec{0}}\right] +\ballostd{\epsilon'}{\vec{0}}.
\]
Thus, we have the following inequalities:
\begin{align*}
    m\left(Y+\ballostd{\epsilon}{\vec{0}}\right) &= m\left(\left[Y + \ballostd{\epsilon-\epsilon'}{\vec{0}}\right] +\ballostd{\epsilon'}{\vec{0}}\right) \tag{Above} \\
    &\geq 
    \left( 
        m\left(Y + \ballostd{\epsilon-\epsilon'}{\vec{0}}\right)^{\frac1d}
        +
        m\left(\ballostd{\epsilon'}{\vec{0}}\right)^{\frac1d}
    \right)^d
    \intertext{The above comes from the \namefullref{generalized-brunn-minkowski-inequality} noting that as demonstrated above, both sets $Y + \ballostd{\epsilon-\epsilon'}{\vec{0}}$ and $\ballostd{\epsilon'}{\vec{0}}$ are open and thus measurable. We continue.}
    &\geq
    \left( 
        m_{out}\left(Y\right)^{\frac1d}
        +
        m\left(\ballostd{\epsilon'}{\vec{0}}\right)^{\frac1d}
    \right)^d
    \intertext{The above inequality comes from the definition of the outer measure of $Y$ being the infimum of the measures of all measurable supersets of $Y$. Since $Y\subseteq Y+\ballostd{\epsilon'}{\vec{0}}$, we get the inequality above. Continuing, we have the following:}
    &=
    \left( 
        m_{out}\left(Y\right)^{\frac1d}
        +
        m\left(\epsilon'\cdot\ballostd{1}{\vec{0}}\right)^{\frac1d}
    \right)^d
    \tag{Scaling of norm-based balls} \\
    &=
    \left( 
        m_{out}\left(Y\right)^{\frac1d}
        +
        \left[(\epsilon')^d \cdot m\left(\ballostd{1}{\vec{0}}\right)\right]^{\frac1d}
    \right)^d
    \tag{Scaling for Lebesgue measure} \\
    &=
    \left( 
        m_{out}\left(Y\right)^{\frac1d}
        +
        \epsilon' \cdot (\unitballmeasurestd)^{\frac1d}
    \right)^d
    \tag{Algebra and $\unitballmeasurestd\defeq m\left(\ballostd{1}{\vec{0}}\right)$}
\end{align*}
Since the inequality above holds for all $\epsilon'\in(0,\epsilon)$, it must also hold in the limit (keeping $d$ and $Y$ fixed):
\[
m\left(Y+\ballostd{\epsilon}{\vec{0}}\right) \geq \lim_{\epsilon'\to\epsilon}\left[\left(m_{out}\left(Y\right)^{\frac1d}+\epsilon' \cdot (\unitballmeasurestd)^{\frac1d}\right)^d\right] = \left(m_{out}\left(Y\right)^{\frac1d}+\epsilon \cdot (\unitballmeasurestd)^{\frac1d}\right)^d
\]
which concludes the proof.
\end{proof}


\restatableOuterMeasureBrunnMinkowskiBoundLInfty
\begin{proof}
This follows immediately from \Autoref{outer-measure-brunn-minkowsi-bound} noting that for the $\ell_\infty$ norm, $\unitballmeasureinf=2^d$.
\end{proof}

\section{Proof of the Folklore \nameref*{lower-bound-cover-number-Rd}}
\label{appendix:measure-theory}

\renewcommand{\A}{\mathcal{A}}
\newcommand{\FF}{\mathcal{F}}

In this appendix, we provide a proof of the known \nameref{lower-bound-cover-number-Rd} restated below.

\restatableLowerBoundCoverNumberRd*

To prove it, we build up the proof via three results: \Autoref{exact-measure-of-multiplicity} handles a general measure space where each point is covered exactly the same number of times and shows what the total measure of the covering sets must be. \Autoref{upper-bound-measure-of-multiplicity} uses this as a lemma to prove what the total measure of the covering sets must be even if points are not all covered the same number of times. \Autoref{lower-bound-cover-number} changes the perspective and focuses on how many times some point must be covered given the total measure of the covering sets. Finally, \numberfullref{lower-bound-cover-number-Rd} follows as an immediate corollary for $\R^d$ specifically.

Throughout this section, we use the word ``countable'' to mean finite or countably infinite.

\begin{fact}[Interchange of Countable Sums with Non-negative Terms]\label{interchange-of-countable-sums}
If $I,J$ are countable sets, and $a_{i,j}\geq0$ for all $(i,j)\in I\times J$, then
\[
    \sum_{i\in I}\sum_{j\in J}a_{i,j} = \sum_{j\in J}\sum_{i\in I}a_{i,j}
\]
\end{fact}
\begin{proof}
This can be proved directly via basic analysis methods if $I$ and $J$ are assumed to be $\N$ and the definition of the infinite sum as a limit of finite sums is used. Alternatively, viewing the summation as an integral over a countable measure space, this can be viewed as an immediate corollary to Tonelli's theorem.
\end{proof}


\begin{lemma}[Exact Measure of Multiplicity]\label{exact-measure-of-multiplicity}
Let $n\in\N$. Let $X$ be a measurable set in some measure space (the measure being denoted by $\mu$) and $\A$ a countable family of measurable subsets of $X$ such that for each $x\in X$, $x$ belongs to {\em exactly} $n$ members of $\A$. Then
\[
\sum_{A\in\A}\mu(A) = n\cdot\mu(X).
\]
\end{lemma}
\begin{proof}
We note that if $n=0$, then the statement is trivially true because $\A$ is either empty or contains just the empty set, and in either case $\sum_{A\in\A}\mu(A) = 0 = 0\cdot\mu(X)$ if we use the standard convention that the empty sum is $0$.

For any $\FF\subseteq\A$, let $G_{\FF}=\bigcap_{A\in\FF}A$ noting that this is a countable intersection of measurable sets, so it is measurable (mnemonically, the $G$ represents an intersection as it does in the notation for $G_{\delta}$ sets).

Let $\binom\A n$ denote all subsets of $\A$ of size $n$ noting that because $\A$ is countable, so is $\binom\A n$. Observe that for distinct $\FF,\FF'\in\binom\A n$, the sets $G_{\FF}$ and $G_{\FF'}$ are disjoint because
\[
    G_{\FF}\cap G_{\FF'} = \left(\bigcap_{A\in\FF}A\right) \cap \left(\bigcap_{A\in\FF'}A\right) = \bigcap_{A\in\FF\cup\FF'}A
\]
and since $\FF$ and $\FF'$ are distinct and each contain $n$ items, $\abs{\FF\cup\FF'}\geq n+1$, and by assumption no point in $X$ belongs to $n+1$ members, so $\bigcap_{A\in\FF\cup\FF'}A=\emptyset$. Furthermore, for each $x\in X$, since $x$ belongs to exactly $n$ members $A_1,\ldots,A_n$ of $\A$, taking $\FF=\set{A_1,\ldots,A_n}$ we have $x\in G_{\FF}$ which shows that $\set{G_{\FF}\colon \FF\in\binom\A n}$ is a partition of $X$ into countably many measurable sets (allowing that some $G_{\FF}$ might be empty).

The last observation we need is that for any $\FF\in\binom\A n$ and any $A\in\A$, it holds that if $A\in\FF$, then $A\supseteq G_{\FF}$ and if $A\not\in\FF$ then $A\cap G_{\FF}=\emptyset$. To see this, note that for any $x\in G_{\FF}$, $x$ belongs to each of the $n$ members of $\FF\subseteq\A$, and since by assumption $x$ belongs to exactly $n$ members of $\A$, it does not belong to any other members of $\A$.

Now we have the following chain of equalities:
\begingroup
\allowdisplaybreaks
\begin{align*}
    \sum_{A\in\A}\mu(A) &= \sum_{A\in\A}\mu(A\cap X) \tag{$A\subseteq X$ so $A\cap X=A$} \\
    &= \sum_{A\in\A}\mu\left(A\cap\left[\bigsqcup_{\FF\in\binom\A n}G_{\FF}\right]\right) \tag{Set equality; the $G_{\FF}$ partition $X$} \\
    %
    &= \sum_{A\in\A}\mu\left(\bigsqcup_{\FF\in\binom\A n}\left[A\cap G_{\FF}\right]\right) \tag{Set equality} \\
    %
    &= \sum_{A\in\A}\left[\sum_{\FF\in\binom\A n}\mu\left(A\cap G_{\FF}\right)\right] \tag{Countable additivity of measures} \\
    &= \sum_{\FF\in\binom\A n}\left[\sum_{A\in\A}\mu\left(A\cap G_{\FF}\right)\right] \tag{Interchange sums by \Autoref{interchange-of-countable-sums}} \\
    &= \sum_{\FF\in\binom\A n}\left[\sum_{A\in\A}
    \begin{cases}
        \mu\left(A\cap G_{\FF}\right) = \mu\left(G_{\FF}\right) & A\in\FF \\
        \mu\left(A\cap G_{\FF}\right) = \mu(\emptyset) = 0 & A\not\in\FF
    \end{cases}
    \right] \tag{Previous paragraph} \\
    &= \sum_{\FF\in\binom\A n}\left[\sum_{A\in\FF}\mu\left(G_{\FF}\right)\right] \tag{Remove $0$ terms from summation} \\
    &= \sum_{\FF\in\binom\A n}\left[n\cdot\mu\left(G_{\FF}\right)\right] \tag{$\abs{\FF}=n$} \\
    &= n\sum_{\FF\in\binom\A n}\mu\left(G_{\FF}\right) \tag{Linearity of summation} \\
    &= n\cdot\mu\left(\bigsqcup_{\FF\in\binom\A n}G_{\FF}\right) \tag{Countable additivity of measures} \\
    &= n\cdot\mu\left(X\right) \tag{Set equality; the $G_{\FF}$ partition $X$} \\
\end{align*}
\endgroup
This proves the result.
\end{proof}


\begin{lemma}[Upper Bound Measure of Multiplicity]\label{upper-bound-measure-of-multiplicity}
Let $n\in\N$. Let $X$ be a measurable set in some measure space (the measure being denoted by $\mu$) and $\A$ a countable family of measurable subsets of $X$ such that for each $x\in X$, $x$ belongs to {\em at most} $n$ members of $\A$. Then
\[
\sum_{A\in\A}\mu(A) \leq n\cdot\mu(X).
\]
\end{lemma}
\begin{proof}
As in the last proof, for any $\FF\subseteq\A$, let $G_{\FF}=\bigcap_{A\in\FF}A$ noting that this is a countable intersection of measurable sets, so it is measurable (mnemonically, the $G$ represents an intersection as it does in the notation for $G_{\delta}$ sets). And for any $k\in[n]\cup\set{0}$, let $\binom\A k$ denote all subsets of $\A$ of size $k$ noting that because $\A$ is countable, so is $\binom\A k$.

For each $k\in[n]\cup\set{0}$, let
\begin{align*}
    S_k &= \set{x\in X\colon \text{$x$ belongs to {\em exactly} $k$ members of $\A$}} \\
    S_k' &= \set{x\in X\colon \text{$x$ belongs to {\em at least} $k$ members of $\A$}}
\end{align*}
We will show that $S_k$ and $S_k'$ are measurable.

To show that the $S_k'$ are measurable, note that for any $k\in[n]\cup\set{0}$, $S_k'$ can be expressed as $S_k' = \bigcup_{\FF\in\binom\A k}G_{\FF}$. This is because for any $x\in X$, if $x$ belongs to at least $k$ members of $\A$, then there is a subset $\FF\subseteq\A$ with $\abs{\FF}=k$ such that $x\in\bigcap_{A\in\FF}A=G_{\FF}$. Conversely, if $x\in\bigcap_{A\in\FF}A=G_{\FF}$, then there is some $\FF\in\binom\A k$ (i.e. some $\FF\subseteq\A$ with $\abs{\FF}=k$) such that $x\in G_{\FF}=\bigcap_{A\in\FF}A$, so $x$ belongs to at least $k$ members of $\A$. Thus, since $\A$ is countable, so is $\binom\A k$ (for each $k$) implying that each $S_k'$ is a countable union of measurable sets, so is itself measurable.

To show the measurability of each $S_k$, first consider $k=n$. Observe that $S_n=S_n'$ because by assumption each $x\in X$ belongs to at most $n$ members of $\A$, so it belongs to exactly $n$ members if and only if it belongs to at least $n$ members. Thus, $S_n$ is also measurable.

Now for $k\in[n-1]\cup\set{0}$ observe that $S_k=S_k'\setminus S_{k+1}'$ because some $x\in X$ belongs to exactly $k$ members of $\A$ if and only if it belongs to at least $k$ members and does not belong to at least $k+1$ members of $\A$. Thus, for $k\in[n-1]\cup\set{0}$, $S_k$ is the set difference of two measurable sets, so is itself measurable.

Finally, note that $\set{S_k\colon k\in[n]\cup\set{0}}$ is a partition of $X$ (allowing the possibility that some $S_k$ are empty) because every point of $x$ belongs to some number of members of $\A$, and that number is (by assumption) between $0$ and $n$ inclusive.

Now we have the following chain of inequalities:
\begingroup
\allowdisplaybreaks
\begin{align*}
    \sum_{A\in\A}\mu(A) &= \sum_{A\in\A}\mu(A\cap X) \tag{$A\subseteq X$ so $A\cap X=A$} \\
    &= \sum_{A\in\A}\mu\left(A\cap\left[\bigsqcup_{k\in[n]\cup\set{0}}S_k\right]\right) \tag{Set equality; the $S_k$ partition $X$} \\
    &= \sum_{A\in\A}\mu\left(\bigsqcup_{k\in[n]\cup\set{0}}\left[A\cap S_k\right]\right) \tag{Set equality} \\
    &= \sum_{A\in\A}\left[\sum_{k\in[n]\cup\set{0}}\mu\left(A\cap S_k\right)\right] \tag{Countable additivity of measures} \\
    &= \sum_{k\in[n]\cup\set{0}}\left[\sum_{A\in\A}\mu\left(A\cap S_k\right)\right] \tag{Interchange sums by \Autoref{interchange-of-countable-sums}} \\
    &= \sum_{k\in[n]\cup\set{0}}\left[k\cdot\mu\left(S_k\right)\right] \tag{By \Autoref{exact-measure-of-multiplicity}; see details below} \\
    &= \sum_{k\in[n]}\left[k\cdot\mu\left(S_k\right)\right] \tag{$k=0$ term is $0$} \\
    &\leq \sum_{k\in[n]}\left[n\cdot\mu\left(S_k\right)\right] \tag{$k\leq n$} \\
    &= n\sum_{k\in[n]}\left[\mu\left(S_k\right)\right] \tag{Linearity of summation} \\
    &= n\cdot\mu\left(\bigsqcup_{k\in[n]}S_k\right) \tag{Countable additivity of measures} \\
    &\leq n\cdot\mu\left(X\right) \tag{Set inequality; the $S_k$ partition $X$, but $S_0$ is missing from the union}
\end{align*}
\endgroup

After justifying the use of \Autoref{exact-measure-of-multiplicity}, this completes the proof. For each $k\in[n]\cup\set{0}$, let $X_k=S_k$ and $\A_k=\set{A\cap S_k\colon A\in\A}$. Then observe that for each $x\in X_k=S_k$, by the definition of $S_k$, $x$ belongs to exactly $k$ members of $\A$, and thus (since it also belongs to $S_k$) belongs to exactly $k$ members of $\A_k$. Applying \Autoref{exact-measure-of-multiplicity} once for each $k$ with $X=X_k$ and $\A=\A_k$ shows that 
\[
 \sum_{A\in\A}\mu(A\cap S_k) = \sum_{A'\in \A_k}\mu(A') = k\cdot\mu(X_k) = k\cdot\mu(S_k)
\]
(the middle equality is where \Autoref{exact-measure-of-multiplicity} was applied). This is what we claimed in the long chain of equalities above and completes the proof.
\end{proof}


\begin{restatable}[Lower Bound Cover Number]{corollary}{restatableLowerBoundCoverNumber}\label{lower-bound-cover-number}
\renewcommand{\A}{\mathcal{A}}
Let $X$ be a measurable set in some measure space (the measure being denoted by $\mu$) such that $0<\mu(X)<\infty$. Let $\A$ be a countable family of measurable subsets of $X$ such that $\sum_{A\in\A}\mu(A)<\infty$. Then there exists $x\in X$ such that $x$ belongs to at least $\ceil{\frac{\sum_{A\in\A}\mu(A)}{\mu(X)}}$-many members of $\A$.
\end{restatable}
\begin{proof}

First observe that by hypothesis, $\ceil{\frac{\sum_{A\in\A}\mu(A)}{\mu(X)}}$ is finite. Suppose for contradiction that each $x\in X$ belongs to strictly less than $\ceil{\frac{\sum_{A\in\A}\mu(A)}{\mu(X)}}$-many members of $\A$. Let $n=\ceil{\frac{\sum_{A\in\A}\mu(A)}{\mu(X)}}-1$ (noting that $n<\frac{\sum_{A\in\A}\mu(A)}{\mu(X)}$). Then each $x\in X$ belongs to at most $n$-many members of $A$, so we have
\begin{align*}
    \sum_{A\in\A}\mu(A) &\leq n\cdot\mu(X) \tag{\Autoref{upper-bound-measure-of-multiplicity}} \\
    &< \frac{\sum_{A\in\A}\mu(A)}{\mu(X)}\mu(X) \tag{$0<\mu(X)<\infty$ and $n<\frac{\sum_{A\in\A}\mu(A)}{\mu(X)}$} \\
    &= \sum_{A\in\A}\mu(A)
\end{align*}
which is a contradiction.
\end{proof}

\begin{remark}\label{infinite-cover-case}
In \Autoref{lower-bound-cover-number} above, it was important that we required $\sum_{A\in\A}\mu(A)$ to be finite. If we allowed it to be infinite, then the claim would have been that there was some $x\in X$ belonging to infinitely many members of $\mathcal{A}$, but this is in general not true (see \Autoref{harmonic-cover-example} below). Nonetheless, it is true (and a straightforward corollary of the above) that if $\sum_{A\in\A}\mu(A)=\infty$, then for any $n\in\N_0$, there exists a point $x_n\in X$ that is contained in at least $n$-many sets of $\A$. The distinction is that this point might have to depend on the choice of $n$.
\end{remark}

\begin{temporary}[Harmonic Cover of Open Unit Interval]\label{harmonic-cover-example}
Let $X=(0,1)$ be equipped with the Borel or Lebesgue measure $\mu$. Let $\mathcal{A}=\set{(0,\tfrac1i): i\in\N}$. Then $\sum_{A\in\A}\mu(A)=\sum_{i\in\N}\tfrac1i=\infty$. For any $n\in\N$, we can consider the point $x_n=\tfrac1{n+1}$ which is contained in $(0,\tfrac1i)$ for $i\in[n]$ and not for any other $i$, so it belongs to exactly $n$ sets in $\A$.

However, no point in $X$ belongs to infinitely many sets in $\A$. To see this, consider an arbitrary point $x\in X=(0,1)$. Then for sufficiently large $i\in\N$, $x\not\in(0,\tfrac1i)$ so $x$ belongs to only finitely many members of $\A$.
\end{temporary}

The prior three results have been stated in typical measure theory notation, but in the body of the paper we present \Autoref{lower-bound-cover-number} as follows for $\R^d$ specifically with notation matching what is used elsewhere in the paper.

\hypertarget{hypertarget-lower-bound-cover-number-Rd}{}
\restatableLowerBoundCoverNumberRd
\begin{proof}
This follows trivially from \Autoref{lower-bound-cover-number} and \Autoref{infinite-cover-case}.
\end{proof}


\end{document}